\documentclass[journal]{article}
\usepackage{arxiv}

\usepackage{amssymb,amsmath,amsthm,amsthm}
\allowdisplaybreaks 
\usepackage{amsbsy}
\usepackage{mathtools}

\usepackage[english]{babel}
\usepackage{accents}
\usepackage{dsfont}
\usepackage{cancel}
\usepackage{bm}
\usepackage{booktabs}
\usepackage{etoolbox}
\apptocmd{\sloppy}{\hbadness 10000\relax}{}{}

\usepackage[shortlabels]{enumitem}

\usepackage{graphicx}
\usepackage[dvipsnames,x11names]{xcolor}

\definecolor{azure}{rgb}{0.0, 0.5, 1.0}
\definecolor{darkred}{rgb}{0.8, 0.0, 0.0}
\definecolor{DarkBlueNavy}{RGB}{0, 0, 128}
\definecolor{nicegreen}{RGB}{76, 175, 80}
\usepackage{hyperref}
\hypersetup{
colorlinks,
linkcolor={nicegreen},
citecolor={azure},
urlcolor={DarkBlueNavy}
}
\makeatletter
\newcommand{\conditem}[2]{%
  \item[#2]%
  \protected@edef\@currentlabel{#2}%
  \label{#1}%
}
\makeatother
\usepackage[style=numeric,sortcites,giveninits=true,maxbibnames=99]{biblatex}
\newtheorem{theorem}{Theorem}[section]
\newtheoremstyle{nonitalic} 
{3pt} 
{3pt} 
{\upshape} 
{} 
{\bfseries} 
{.} 
{.5em} 
{} 

\theoremstyle{nonitalic}
\theoremstyle{nonitalic}
\newtheorem{remark}[theorem]{Remark}

\theoremstyle{plain} 
\newtheorem{definition}[theorem]{Definition}
\newtheorem{proposition}[theorem]{Proposition}
\newtheorem{lemma}[theorem]{Lemma}
\newtheorem{corollary}[theorem]{Corollary}

\newtheorem*{theorem*}{Theorem}

\numberwithin{equation}{section}

\newcommand{\N}{\mathbb{N}}

\newcommand{\R}{\mathbb{R}}

\newcommand{\dom}{\operatorname{dom}}

\newcommand{\Deltadir}{\Delta_{\textnormal{dir}}}
\newcommand{\Deltadirn}{\Delta_{\textnormal{dir},n}}
\newcommand{\Amin}{\mathcal{A}_{n}}

\makeatletter
\newcommand{\sbullet}{%
	\hbox{\fontfamily{lmr}\fontsize{.6\dimexpr(\f@size pt)}{0}\selectfont\textbullet}}
\DeclareRobustCommand{\mathbullet}{\accentset{\sbullet}}
\makeatother
\newcommand{\Hm}[1]{\leavevmode{\marginpar{\tiny%
			$\hbox to 0mm{\hspace*{-0.5mm}$\leftarrow$\hss}%
			\vcenter{\vrule depth 0.1mm height 0.1mm width \the\marginparwidth}%
			\hbox to 0mm{\hss$\rightarrow$\hspace*{-0.5mm}}$\\\relax\raggedright
			#1}}}
\begin{document}

\title{Semilinear diffusion equations on infinite graphs:\\ the dissipative and Lipschitz cases}
\date{}
\author{Elvise Berchio\\
Department of Mathematical Sciences "G.L. Lagrange"\\ Politecnico di Torino\\
Torino 10129 (Italy)\\
\textit{elvise.berchio@polito.it}
\And 
Davide Bianchi\\
	School of Mathematics (Zhuhai)\\
	Sun Yat-sen University\\
	Zhuhai 518055 (China)\\
    \textit{bianchid@mail.sysu.edu.cn}\\
\And
Alberto G. Setti\\
Dipartimento di Scienza e Alta Tecnologia\\
	Universit\`a dell'Insubria\\
	Como 22100 (Italy)\\
\textit{alberto.setti@uninsubria.it}\\
\And
Maria Vallarino\\
Department of Mathematical Sciences "G.L. Lagrange"\\ Politecnico di Torino\\
Torino 10129 (Italy)\\
\textit{maria.vallarino@polito.it}
}

\setcounter{tocdepth}{2}

\maketitle

\begin{abstract}
We study a class of semilinear diffusion equations on infinite, connected, weighted graphs, focusing on two types of nonlinearities: monotone decreasing and Lipschitz continuous. Under minimal structural assumptions on the graph, we establish existence, uniqueness, and regularity of mild solutions for initial data in $\ell^p$ spaces, with $1\leq p<\infty$. Our approach relies on time discretization via an implicit Euler scheme and an exhaustion technique using Dirichlet subgraphs. As a by-product, we obtain existence and uniqueness results for a related time-independent equation. Finite-time extinction and positivity for solutions under a specific forcing term are also proved. 

\end{abstract}

\keywords{graphs; semilinear diffusion equation; existence of solutions; uniqueness of solutions.}
{\textbf{\textit{ 2020 AMS Subject Classification:}}
35K55, 35A01, 05C22, 05C63, 47H06.}

\section{Introduction}

Given a connected graph $G=(X,w,\kappa,\mu)$ (see Section \ref{sec:notation} for the precise definition), we consider the problem

\begin{equation}\label{model_problem}\tag*{(Model Problem 1)}
    \begin{cases}
		\partial_t u(t,x) + \Delta u(t,x) =  f(u(t,x))+h(t,x)  &   \quad (t,x) \in (0,T)\times X,\\
		u(0,x) = u_0(x) &  \quad x \in X,
	\end{cases}
	\end{equation}
where $\Delta$ is the (negative) Laplacian associated with the graph and $T>0$. A simple transformation allows to show the equivalence of \ref{model_problem} on subgraphs to a related initial boundary value problem which constitutes our \ref{model_problem2}, see Section \ref{MP2}.

 This class of semilinear diffusion equations on graphs has many different applications in 
network dynamics~\cite{proskurnikov2017tutorial,proskurnikov2018tutorial}, 
population and epidemic diffusion models~\cite{pastor2015epidemic,allen2017primer}, 
opinion formation and consensus problems on social networks~\cite{acemouglu2013opinion}, 
and more generally in nonlinear diffusion processes on discrete structures, where the underlying graph encodes the geometry of interactions. 
Further applications include chemical and biological diffusion models, where reactions occur at the nodes while transport follows the graph topology~\cite{van2009stochastic,mckay2014stochastic}, 
as well as graph-based learning and image processing, where semilinear diffusions are exploited for tasks such as denoising, segmentation, and semi-supervised classification~\cite{bertozzi2012diffuse,elmoataz2008nonlocal}. 

Due to  the many   applications, there is a rapidly growing literature on semilinear and nonlinear equations on graphs. As a representative example we mention~\cite{bianchi2022generalized,ma2022porous, GLY16a, grigor2016yamabe, HX23, LSZ23, LY21,  KS18,HM15,monticelli2025nonexistence}.

In particular, global existence and blow-up of solutions for parabolic equations has been studied by several authors in~\cite{grillo2024blow,punzo2025semilinear,wu2024blow,wu2021blow,monticelli2024nonexistence,lin2017existence,yong2018blow,slavik2017well,hu2026life}. 

 In this work we study the case where $u_0 \in \ell^p(X,\mu)$, $h \in L^1_{\textnormal{loc}}\left([0,T]; \ell^p\left(X,\mu\right)\right)$ ($1\leq p <\infty$) and $f: \R \rightarrow \R$ is either monotone or Lipschitz continuous, namely it satisfies one of the following conditions: 
\begin{enumerate}[label=\textbf{(F\arabic*)},ref=\textbf{(F\arabic*)}]
    \item\label{F1:monotonicity} $f$ is continuous and monotone decreasing with $f(0)=0$;
    \item\label{F2:Lipschitz} $f$ is uniformly Lipschitz with constant $L>0$ and $f(0)=0$.
\end{enumerate}
Condition~\ref{F1:monotonicity} requires the forcing term \(f\) to act as a sink, i.e., \(f(u)\) tends to dissipate the dynamics for \(u>0\). A typical example is  
\begin{equation}\label{specific}
f(t) = -t|t|^{q-1}, \quad q>0\,.
\end{equation}
In this case, the equation in~\ref{model_problem} is known as the \textit{\(w\)-heat equation with absorption}, where \(w\) denotes the edge-weight function that encodes the graph topology (see Section~\ref{sec:notation} for details).  For this specific forcing term in the finite-graph setting, with the addition of homogeneous Dirichlet boundary conditions, uniqueness, extinction, positivity and asymptotic behavior of solutions were proved in \cite{chung2011extinction} by applying spectral analysis (using eigenvalues and eigenvectors of the Laplacian matrix) and by exploiting energy functionals (see also \cite{chung2017plaplacian} for the $p-$Laplacian version). For the dissipative case in \(\R^N\) with \(f=-\beta\), where \(\beta\) is a maximal monotone operator, see~\cite[Chapter~5]{barbu2010nonlinear}.  

By contrast, Condition~\ref{F2:Lipschitz} is characteristic of saturating nonlinear forcing terms, which frequently arise in biological models (see, e.g.,~\cite{cantrell2004spatial}). For a detailed treatment of this case, see also~\cite[Chapter~31.6]{zeidler2013nonlinear}.

When Condition \ref{F1:monotonicity} or \ref{F2:Lipschitz} hold, we establish existence, uniqueness, and regularity of solutions to \ref{model_problem} for $\ell^p$ data, under fairly weak assumptions on $X$ (see Theorems \ref{thm:main2} and \ref{thm:regularity}). The proof is carried out by discretizing time and constructing approximate solutions via an implicit Euler scheme. A crucial step involves demonstrating the accretivity or $\omega$-accretivity of the operator $\Delta-f$, which ensures convergence of the approximations to a mild solution. This is achieved by introducing a suitable increasing exhaustion of the graph through finite node subsets, each equipped with a corresponding Dirichlet subgraph structure. Notably, this method does not require any information about heat kernel or spectral properties of $\Delta$, making it particularly effective for general graph structures and for applications. 

As a by-product of our approach, we also obtain existence, uniqueness, and an a priori estimate for the time-independent equation:
\begin{equation*}
  \Delta u (x)+ \alpha u(x)=f(u(x))+ \alpha g(x)\quad  \quad x \in X
\end{equation*}
with $g \in \ell^p(X,\mu)$ and $\alpha>0$ if condition \textnormal{\ref{F1:monotonicity}} holds ($\alpha>L$, if condition \textnormal{\ref{F2:Lipschitz}} holds), see Theorem \ref{th: F1}. 
The equation above resembles  the one studied in \cite{grigor2017existence} and in a number of subsequent papers, see e.g., \cite{yang2023existence} and references therein; however, there are substantial differences in both the assumptions on the graph and the nature of the nonlinearity which reflect the different strategies employed. In particular, the hypotheses in \cite{grigor2017existence} are specifically designed to provide a structural framework suitable for the application of critical point theory, especially the mountain pass theorem.

The paper is organized as follows: Section \ref{sec:notation} provides the necessary preliminaries and notation, including the formal definition of graphs, function spaces, and the graph Laplacian, as well as the concept of Dirichlet subgraphs. In this section, we also introduce the second model problem on subgraphs \ref{model_problem2} and we show how it can be reformulated as \ref{model_problem}. In Section \ref{main}, we first introduce the definitions of classical, strong, and mild solutions to~\ref{model_problem}, and then present the main results concerning the existence, uniqueness, and regularity of mild solutions. Section~\ref{auxiliary section} is devoted to the analysis of a related time-independent equation associated with the semilinear diffusion problem. We establish existence, uniqueness, and a priori estimates for solutions under Condition \ref{F1:monotonicity} or \ref{F2:Lipschitz}. We also introduce a dense subset of the operator domain and prove accretivity properties that are crucial for the construction of mild solutions in the time-dependent setting. Section \ref{s:proofs} contains the proofs of the main theorems. Finally, Section \ref{s: specificnonlinearity} includes a parabolic comparison principle and a discussion of positivity and finite-time extinction for mild solutions to \ref{model_problem} in the particular case where $f$ is as in \eqref{specific} and $h=0$.

\section{Preliminaries and notations}\label{sec:notation}
In this section, we introduce the preliminary notation and background material on graphs and operators that will be used throughout the paper.  

Let $\psi \colon D \subseteq \R \to \R$ be a function. With a slight abuse of notation, we continue to denote by $\psi$ its canonical extension to the function space
\[
\dom(\psi) =\{ u \colon Z \to \R \mid u(z) \in D \;\; \forall\, z\in Z\}\subseteq \{u \colon Z \to \R \},
\]
where $Z$ is a generic set. Explicitly,
\[
\psi \colon \dom(\psi) \to \{u \colon Z \to \R \}, \qquad (\psi u)(z) \coloneqq \psi(u(z)).
\]

For a detailed introduction to the graph framework adopted here, we refer to~\cite{keller2021graphs}. We now begin with the definition of a graph.
\begin{definition}[Graph]
A \textit{graph} is a quadruple $G=(X,w,\kappa,\mu)$ given by
\begin{itemize}
	\item  a countable set of \emph{nodes} $X$;
	\item a nonnegative \emph{edge-weight} function $w\colon X\times X \to [0,\infty)$;
	\item  a nonnegative \emph{killing term} $\kappa \colon X \to [0,\infty)$;
	\item a positive \emph{node measure} $\mu \colon X \to (0,\infty)$
\end{itemize}
where the edge-weight function $w$ satisfies:
\begin{enumerate}[label={\upshape(\bfseries A\arabic*)},wide = 0pt, leftmargin = 3em]
	\item\label{assumption:symmetry} Symmetry: $w(x,y)=w(y,x)$ for every $x,y \in X$;
	\item\label{assumption:loops} No loops: $w(x,x)=0$ for every $x \in X$;
	\item\label{assumption:degree} Finite sum: $\sum_{y\in X} w(x,y) < \infty$ for every $x \in X$.
\end{enumerate}
\end{definition}

If the cardinality of the node set is finite, i.e., $|X|<\infty$, then $G$ is called a \emph{finite graph}, otherwise, $G$ is called an \emph{infinite graph}. We say that two nodes $x$ and $y$ in $X$ are neighbors or connected by an edge, and we write $x\sim y$, if and only if $w(x,y)>0$. We will also say, as in standard graph theory, that $\{x,y\}$ is an edge of the graph. The non-zero values $w(x,y)$ of the edge-weight function $w$ are called \emph{weights} associated with the edge $\{x,y\}$. On the other hand, if $w(x,y)=0$, then we will write $x\nsim y$ meaning that $x$ and $y$ are not connected by an edge. A \emph{walk} is a (possibly infinite) sequence of nodes $\{x_{i}\}_{i\geq 0}$ such that $x_{i}\sim x_{i+1}$. A \emph{path}
is a walk with no repeated nodes.  A graph is \emph{connected}
if there is a finite walk connecting every pair of nodes, that is, for any pair of nodes $x, y$ there is a finite walk such that $x = x_{0}\sim x_{1}\sim\cdots\sim x_{n}=y$. Moreover, we will say that a subset $Y \subseteq X$  is connected if for every pair of nodes $x,y \in Y$ there exists a finite walk connecting $x$ and $y$ all of whose nodes are in $Y$. A subset $Y \subseteq X$ is a \emph{connected component} of $X$ if $Y$ is maximal with respect to inclusion.

We define the \textit{degree} $\operatorname{deg}$ of a node as
\begin{equation*}
\operatorname{deg}(x) = \sum_{y\in X} w(x,y) + \kappa(x).
\end{equation*}
 The presence of $w$ and $\kappa$ gives a one-to-one correspondence between graphs and Markov processes. The edge-weight function $w$ encodes how the process jumps between nodes of $X$ while the killing term $\kappa$ captures the possibility of the process leaving $X$. In other words, the value of $\kappa$ indicates the weight of a connection from that vertex to some
additional point outside of $X$. This motivates the name killing term. For a detailed discussion on the meaning of $w$ and $\kappa$ we refer to \cite[Part 0]{keller2021graphs}.

In the sequel we shall make use of the following conditions on the graph:

\begin{enumerate}[label={}, leftmargin=*, align=left]
  \conditem{IP}{\textbf{(IP)}} for every \textit{infinite path} $\{x_n\}_n$, $\sum_n \mu(x_n)=\infty$;
  \conditem{B}{\textbf{(B)}} the edge degree is \textit{bounded}, i.e., $\sup_{x\in X} \frac{\sum_{y\in X} w(x,y)}{\mu(x)}<\infty$.
\end{enumerate}

	The set of real-valued functions on $X$ is denoted by $C(X)$ and $C_c(X)$ denotes the set of functions on $X$ with finite support. As usual, for $p\in [1,\infty]$ we define the $\ell^p(X,\mu)$ subspaces of $C(X)$ as
	$$
	\ell^p(X,\mu)\coloneqq\begin{cases}
	\left\{ u \in C(X) \mid \sum_{x\in X} |u(x)|^p\mu(x)<\infty \right\} & \mbox{for } 1\leq p<\infty,\\
	\left\{ u \in C(X) \mid \sup_{x\in X}|u(x)|<\infty \right\} & \mbox{for } p =\infty,
	\end{cases}
	$$
	with their norms
	$$
	\|u\|_p\coloneqq \begin{cases}
	\left(\sum_{x\in X} |u(x)|^p\mu(x)\right)^{{1}/{p}} & \mbox{for } 1\leq p<\infty ,\\
	\sup_{x\in X}|u(x)| & \mbox{for } p =\infty.
	\end{cases}
	$$	

In addition to the previous standard definitions, we introduce the following restrictions to the nonnegative/nonpositive cones:
\begin{align*}
\ell^{p,+}\left(X,\mu\right)\coloneqq \ell^{p}\left(X,\mu\right) \cap \left\{ u \in C(X) \mid u\geq 0 \right\}, \quad
\ell^{p,-}\left(X,\mu\right)\coloneqq -\ell^{p,+}\left(X,\mu\right).
\end{align*}

The \emph{formal graph Laplacian}  $\Delta \colon \dom\left(\Delta\right) \subseteq C(X) \to C(X)$ associated to the graph $G=(X,w,\kappa,\mu)$ is defined by
		\begin{subequations}
			\begin{equation}
			\dom\left(\Delta\right)\coloneqq\{ u \in C(X) \mid \sum_{y\in X}w(x,y)|u(y)| < \infty \quad \forall x \in X  \},\label{formal_laplacian1}
			\end{equation}
			\begin{align}
			\Delta u(x)&\coloneqq \frac{1}{\mu(x)}\sum_{y\in X} w(x,y)\left(u(x) - u(y)\right) + \frac{\kappa(x)}{\mu(x)}u(x)\label{formal_laplacian2}\\
            &= \frac{1}{\mu(x)}\left(\operatorname{deg}(x)u(x) - \sum_{y\in X}w(x,y)u(y)\right)\qquad x\in X. \nonumber
			\end{align}
		\end{subequations}
We note that, in general,  $\Delta(C_c(X)) \subseteq \ell^p(X,\mu)$ is not true.  	That is, the formal graph Laplacian does not necessarily map finitely supported functions into $\ell^p(X,\mu)$ spaces. This is a particularly important property that will be useful later, and it is equivalent to the validity of 
\begin{equation}
\begin{cases}\tag{\textnormal{$C_p$}}\label{eq:Delta(C_c)}
 \sum_{y\in X} \frac{w(x,y)^p}{\mu(y)^{p-1}}  < \infty & \mbox{if} \quad 1\leq p < \infty,\\
 \sup_y \frac{w(x,y)}{\mu(y)}  < \infty & \mbox{if} \quad p= \infty,
\end{cases}\qquad \mbox{for every } x\in X.
\end{equation}

Indeed, since for every $x\in X$ 
$$
\Delta \delta_x=-\sum_{y\sim x}\frac{w(x,y)}{\mu(y)}\delta_y + \frac{ \sum_{y\sim x}w(x,y)}{\mu(x)}\delta_x  +\frac{\kappa(x)}{\mu(x)}\delta_x,
$$
it is easy to see that $\Delta \delta_x\in \ell^p(X,\mu)$ if and only if ~\eqref{eq:Delta(C_c)} is satisfied.

\begin{remark}\label{rem:Delta(C_c)_in_ell1}
Note that Condition~\eqref{eq:Delta(C_c)} is always true for $p=1$, due to Assumption~\ref{assumption:degree}. In particular, it always holds that $\Delta (C_c(X)) \subseteq \ell^1(X,\mu)$. Moreover, ~\eqref{eq:Delta(C_c)} is also valid  for every $1<p<\infty$  if $G$ is locally finite or if the measure $\mu$ is uniformly bounded from below. 
\end{remark}

\smallskip

In the paper, we study the nonlinear operator $\mathcal{A}$, obtained by adding the nonlinear operator $F$ to the graph Laplacian~$\Delta$, defined as follows:  
\begin{eqnarray}\label{A}
	&\mathcal{A} := F + \Delta,\\
	&\dom_p(\mathcal{A}) = \{ u \in \ell^p(X,\mu) \mid u \in \dom(\Delta) , \; \mathcal A u \in  \ell^p(X,\mu)\;   \} \notag
\end{eqnarray}
with $Fu(t,x) \coloneq -f(u(t,x))$. Let us remark that, since $f$ is defined on all of $\R$, then $\dom(F) = C(X)$.

As we will see in the following, a key ingredient in our proof is the approximation of the operator $\mathcal{A}$ by a sequence of operators defined on finite subsets of $X$. This approach requires introducing the notion of subgraphs of a graph. We begin by discussing the concept of the interior and two types of boundary for a subset of the node set. Given a graph $G=(X,w,\kappa,\mu)$ and a subset $Y\subset X$, then
	$$
	\mathring{Y}\coloneqq\left\{ y \in Y \mid y \nsim x \mbox{ for every } x \in X\setminus Y \right\}
	$$
	is called the \emph{interior} of $Y$ and the elements of $\mathring{Y}$ are called \emph{interior nodes} of $Y$. On the other hand, we define
	\begin{align*}
	&\mathring{\partial} Y\coloneqq\left\{ y \in Y \mid y \sim x  \mbox{ for some }  x \in X\setminus Y \right\} \qquad\textnormal{interior boundary},\\
	&\mathbullet{\partial} Y\coloneqq\left\{ x \in X\setminus Y \mid x \sim y  \mbox{ for some }  y \in  Y \right\} \qquad\textnormal{exterior boundary}.
	\end{align*}

	\begin{definition}\label{def:induced_subgraph}
 We say that a graph $G'=(Y, w',\kappa', \mu')$ is an \emph{induced subgraph} of $G=(X,w,\kappa,\mu)$, and we write $G'\subset G$, if
	\begin{itemize}
		\item $Y\subset X$;
		\item $w' \equiv w_{|Y\times Y}$;
		\item $\kappa'(y) = \kappa(y)$ for every $y \in\mathring{Y}$;
		\item $\mu' \equiv \mu_{|Y}$,
	\end{itemize}
where $w_{|Y\times Y}$ and $\mu_{|Y}$ denote the restrictions of $w$ and $\mu$ to the sets $Y\times Y$ and $Y$, respectively.
	We call $G$ the \emph{host graph} or the \emph{supergraph}. 
	The corresponding formal graph Laplacian for a subgraph $G'$ is defined according to \eqref{formal_laplacian1} and \eqref{formal_laplacian2} where the quadruple $(X,w,\kappa,\mu)$ is replaced by $(Y,w',\kappa',\mu')$. We say that $G'$ is the \emph{canonical induced subgraph} if $\kappa'=\kappa_{|Y}$.
	\end{definition}
Observe that in the definition above we do not require that $\kappa' \equiv \kappa$ on $\mathring{\partial}Y$. Different choices of $\kappa'$ on $\mathring{\partial} Y$ will produce different subgraphs. 

     \begin{definition}\label{def:dir_subgraph}
      Given a graph $G=(X,w,\kappa,\mu)$ and $Y\subset X$, we define
     \begin{equation}\label{def:dir_potential}
	\begin{cases}
	 b_{\textnormal{dir}}(y)\coloneqq\sum_{z \in  \mathbullet{\partial} Y}w(y,z) = \sum_{z \not \in Y}w(y,z),\\
     \kappa_{\textnormal{dir}}(y) \coloneqq \kappa_{|Y}(y) + b_{\textnormal{dir}}(y).\\
	\end{cases}	
		\end{equation}
The induced subgraph 
		$$
		G_{\textnormal{dir}} \coloneqq\left(Y, w_{|Y\times Y}, \kappa_{\textnormal{dir}}, \mu_{|Y}\right)\subset G
		$$
		is called a \emph{Dirichlet subgraph}.

We note that $b_{\textnormal{dir}} \colon Y \to \R$ is finite because of \ref{assumption:degree}. We call $b_{\textnormal{dir}}$ the \emph{boundary (Dirichlet) weight-function} and  $\kappa_{\textnormal{dir}}$ the \emph{Dirichlet killing term}. Clearly, $b_{\textnormal{dir}}(y)= 0$ for every $y \in \mathring{Y}$. We will denote by $\Delta_{\textnormal{dir}}$ the graph Laplacian of $G_{\textnormal{dir}}$  in order to distinguish it from the graph Laplacian of $G$.
	\end{definition}
The Dirichlet killing term accounts for the edge deficiency of nodes in $G_{\textnormal{dir}}$ compared to the same nodes in $G$.

Let $\boldsymbol{\mathfrak{i}} \colon  C(Y) \hookrightarrow  C(X)$ be the canonical embedding and $\boldsymbol{\pi} \colon C(X) \to C(Y)$ be the canonical projection, i.e.,
\begin{equation}\label{eq:embedding-projection}
	\boldsymbol{\mathfrak{i}}v(x)= \begin{cases}
	v(x) & \mbox{if } x\in Y,\\
	0 & \mbox{if } x\in X\setminus Y,
	\end{cases}
	\qquad
	\boldsymbol{\pi}u= u_{|Y}.
\end{equation}
If $\mathbullet{\partial} Y \neq \emptyset$, in \cite[Lemma A1]{bianchi2022generalized}, it was proved that
	\begin{align}
		& \Deltadir v(y)=\Delta\boldsymbol{\mathfrak{i}}v(y) \mbox{ for } v\in \dom(\Deltadir) \mbox{ and } y \in Y;\label{eq:Delta_dir_property_1}\\
		& \Delta u(y) = \Deltadir\boldsymbol{\pi}u(y) \mbox{ for } u\in \dom\left(\Delta\right)\cap \left\{u \in C(X) \mid u \equiv 0 \mbox{ on } X\setminus Y \right\} \mbox{ and } y \in Y.\label{eq:Delta_dir_property_2}
	\end{align}
	Therefore, the Dirichlet graph Laplacian $\Deltadir$ can be viewed as the restriction of $\Delta$ having imposed Dirichlet conditions on the exterior boundary of $Y$.

\subsection{Model Problem 2}\label{MP2}
 Let $Y\subsetneq X$ be a nonempty proper subset of the discrete node set $X$ and let $\mathbullet{\partial}Y$ be the exterior boundary of $Y$. In this section we show that   the following initial boundary value problem on $Y$ 
\begin{equation}\label{model_problem2}\tag*{(Model Problem 2)}
\begin{cases}
	\partial_t u(t,x) + \Delta u(t,x) =  f(u(t,x)) &  \quad (t,x) \in (0,T)\times Y,\\
 	 u(t,x) = g(t,x) & \quad (t,x) \in [0,T)\times \mathbullet{\partial}Y,\\
	 u(0,x) = u_0(x) & \quad x \in Y\,,
	\end{cases}
\end{equation}
 can be reformulated as \ref{model_problem} by replacing $G$ with a suitable subgraph $G'$ and for a suitable choice of the forcing term $h$ absorbing the Dirichlet BCs.

Indeed, let $u$ be a solution to \ref{model_problem2}.  Making explicit the action of the graph Laplacian on $G$, for every $x \in Y$, $t\in(0,T)$, it holds 
\begin{align*}
	\Delta u(t,x) &= \frac{1}{\mu(x)}\sum_{\substack{y\sim x\\y \in X}} w(x,y)\left(u(t,x) - u(t,y)\right) + \frac{\kappa(x)}{\mu(x)}u(t,x)\\
	&= \frac{1}{\mu(x)}\sum_{\substack{y\sim x\\y \in Y}} w(x,y)\left(u(t,x) - u(t,y)\right) + \frac{1}{\mu(x)}\sum_{\substack{y\sim x\\y \in \mathbullet{\partial} Y}} w(x,y)\left(u(t,x) - g(t,y)\right)+ \frac{\kappa(x)}{\mu(x)}u(t,x)\\
	&= \frac{1}{\mu(x)}\sum_{\substack{y\sim x\\y \in Y}} w(x,y)\left(u(t,x) - u(t,y)\right) +  \left(\frac{\kappa(x)}{\mu(x)}+ \frac{b_{\operatorname{dir}}(x)}{\mu(x)}\right)u(t,x) - h(t,x), 
\end{align*}
where $b_{\operatorname{dir}}$ is as in \eqref{def:dir_potential} and
\begin{equation}\label{eq: hvsg}
h(t,x)= \frac{1}{\mu(x)}\sum_{\substack{y\sim x\\y \in \mathbullet{\partial} Y}} w(x,y)g(t,y).
\end{equation}
Observing that
$$
\frac{1}{\mu(x)}\sum_{\substack{y\sim x\\y \in Y}} w(x,y)\left(u(t,x) - u(t,y)\right) + \left(\frac{\kappa(x)}{\mu(x)}+ \frac{b_{\operatorname{dir}}(x)}{\mu(x)}\right)u(t,x)
$$
defines the action of the graph Laplacian associated with the graph $G':=\left(Y,  w', \kappa_{\operatorname{dir}}, \mu'\right)\subset G$, where $\kappa_{\operatorname{dir}}$ is as in \eqref{def:dir_potential}, $w' = w_{|Y}$ and $\mu'= \mu_{|Y}$, we conclude that  any solution to \ref{model_problem2} solves \ref{model_problem} with $G$ replaced by $G'$ and $h$ as in~\eqref{eq: hvsg}. In particular, $h\equiv 0$ if $g \equiv 0$.

\subsection{Accretivity}

As is well known, a crucial property when dealing with nonlinear diffusions is accretivity. We recall the definition of accretive and $\omega$-accretive operators here below.  

For an overview of accretive operators and nonlinear evolution equations in Banach spaces, we refer the interested reader to the seminal monograph by Bénilan, Crandall, and Pazy~\cite{benilan1988evolution}, or to Barbu's book~\cite{barbu2010nonlinear} for a more modern introduction. The theory is developed in the general setting of multivalued operators; however, in our framework we restrict the study to the single-valued case.
\begin{definition}\label{def: accretivity}
If $(E,\|\cdot\|)$ is a real Banach space and $\mathcal{A}\colon \dom(\mathcal{A})\subseteq E \to E$ is a (not necessarily linear) operator, then $\mathcal{A}$ is said to be \emph{accretive} if $\mathcal{A}$ satisfies:
\begin{equation*}
	\left\|(u - v) + \lambda \left(\mathcal{A}u - \mathcal{A}v\right) \right\|\geq \|u - v\| \quad \mbox{for every}\; u,v \in \dom\left(\mathcal{A}\right)\; \mbox{and for every}\; \lambda >0.
\end{equation*}
Given $\omega\in \R$ the operator $\mathcal A$ is said to be \emph{$\omega$-accretive} if $\mathcal{A} +\omega\operatorname{id}$ is accretive. An accretive operator is said $m$-accretive if $\operatorname{id} + \lambda\mathcal{A}$  is surjective for every $\lambda >0$.
\end{definition}

We note for future use that $\mathcal{A}$ is $\omega$-accretive if and only if 
\begin{equation}
\label{eq:omega-accretivity-1}
\|(u-v)+\lambda (\mathcal{A}u-\mathcal{A}v)\|\geq
(1-\omega \lambda)\|u-v\|,
\end{equation}
for all $u,v\in \dom\mathcal{A}$ and $0\leq\lambda\omega<1
$, see  \cite[Section 2.1]{benilan1988evolution}.

In the special case where $E = \ell^p(X,\mu)$, $1\leq p<\infty$, the above definition admits an equivalent closed form expression. That is, writing
\begin{equation*}
z \coloneq \mathcal{A}u - \mathcal{A}v, \qquad k \coloneqq u - v,
\end{equation*}
then the operator $\mathcal{A}$ is accretive if and only if for every $u, v \in \dom(\mathcal{A})$, $u\neq v$, the following holds
\begin{equation}\label{accretivity2}
\begin{cases}
\; \;  \smashoperator{\sum_{x\in X}} z(x)|k(x)|^{p-1}\operatorname{sgn}(k(x)) \mu(x) \geq 0 &\mbox{for } 1<p<\infty,\\
     & \\
\; \; \smashoperator{\sum_{\substack{x\in X {:} \\ k(x) = 0}}}
|z(x)|\mu(x) + \smashoperator{\sum_{\substack{x\in X {:} \\ k(x) \neq 0} }}
z(x)\operatorname{sgn}(k(x))\mu(x) \geq 0 &\mbox{for } p=1,
\end{cases}
\end{equation}
where 
\[
\operatorname{sgn}(t)=
\begin{cases}
1&\text{if } t>0,\\
0&\text{if } t=0,\\
-1&\text{if } t<0.
\end{cases}
\]
See \cite[Example 2.9]{benilan1988evolution}. 

Next, we recall the following  lemma.
\begin{lemma}\label{lem:positivity}
Let $G$ be a  graph and let $u\in C_c(X)$. Then, for every $1\leq p< \infty,$
\[ \sum_{x\in X}
\Delta u(x)|u(x)|^{p-1}\operatorname{sgn}(u(x))\mu(x)  \geq 0.
\]
In particular, the above inequality holds for $u\in C(X)$ if $G$ is finite. 
\end{lemma}
\begin{proof} 
The statement follows from  \cite[Proposition 2.5]{bianchi2022generalized} for $p>1$ and from \cite[Proposition  B.2]{bianchi2022generalized}  with  $v=0$ and  $\Phi={\rm{id}}$, for $p=1$, noting that in both cases the argument depends on  a Green's formula which is  valid for finitely supported functions.
\end{proof}
The above lemma and \eqref{accretivity2} yield the following proposition.  
\begin{proposition}\label{prop:Delta_accretivity_finite_graphs}
  Let $G$ be a finite graph. Then, the graph Laplacian $\Delta$ on $\ell^p(X,\mu)$ is $m$-accretive for $1\leq p <\infty$.  
\end{proposition}
\begin{proof}
This result is well known and follows from the Hille–Yosida Theorem, since $\Delta$ generates a contractive semigroup on $\ell^p(X,\mu)$ on finite graphs (\cite[Theorem 2.9]{keller2021graphs}).
For the reader’s convenience, we briefly sketch the argument. Accretivity is an immediate consequence of the linearity of $\Delta$ together with \eqref{accretivity2} and Lemma \ref{lem:positivity}. Since accretivity implies the injectivity of ${\rm{id}}+\lambda\Delta $ then  it  is also surjective since it is linear and $C(X)$ is finite dimensional.
\end{proof}

\section{Main results}\label{main}
\subsection{Definition of solutions}

Hereafter, using the definition of the operator $\mathcal A$ given in formula \eqref{A}, we reformulate \ref{model_problem} as follows
\begin{equation}\label{model_problem3}
	\begin{cases}
		\partial_t u(t,x) + \mathcal{A} u(t,x) = h(t,x)  &  \quad (t,x) \in (0,T)\times X,\\
		u(0,x) = u_0(x) & \quad x \in X,
	\end{cases}
\end{equation}
where $h \in L^1_{\rm{loc}}([0,T] ; \ell^p\left(X,\mu\right))$ with $1\leq p <\infty$ and $u_0 \in {\ell^p(X,\mu)}$. 
\begin{definition}[Classical  Solution]\label{def:classical_solution}
	We say that $u \colon [0,T] \to \ell^p\left(X,\mu\right)$ is a \emph{classical solution} of Problem \eqref{model_problem3} if
	\begin{itemize}
		\item $u(t) \in {\dom_p(\mathcal{A})}$ for every $t \in [0,T]$;
		\item $u \in C\left( [0,T]; \ell^p\left(X,\mu\right)\right)\cap C^1\left( (0,T); \ell^p\left(X,\mu\right)\right)$;
		\item $\partial_t u(t,x) + \mathcal{A} u(t,x) = h(t,x) $ for every $t\in (0,T)$;
		\item $u(0)=u_0$.
	\end{itemize}
\end{definition}

\begin{definition}[Strong Solution]\label{def:strong_solution}
	We say that  $u \colon [0,T] \to \ell^p\left(X,\mu\right)$  is a \emph{strong solution} of Problem \eqref{model_problem3} if
	\begin{itemize}
		\item $u(t) \in {\dom_p(\mathcal{A})}$ for every $t \in [0,T]$;
		\item $u \in C\left([0,T]; \ell^p\left(X,\mu\right)\right)\cap W^{1,1}_{\textnormal{loc}}((0,T);\ell^p\left(X,\mu\right))$;
		\item $\partial_t u(t,x) + \mathcal{A} u(t,x) = h(t,x) $ for almost every $t\in (0,T)$;
		\item $u(0)=u_0$.
	\end{itemize}
\end{definition}

In the paper we will mainly deal with \emph{mild solutions} (see e.g., \cite{barbu2010nonlinear} and \cite{benilan1988evolution}), whose definition relies on the definition of $\epsilon$-approximate solutions given below.

\begin{definition}[$\epsilon$-discretization]\label{def:epsilon-discretization}
	Given a time interval $[0,T]$, with $T<+\infty$, and a forcing term $h \in L^1_{\rm{loc}}([0,T] ; \ell^p\left(X,\mu\right))$, we define a partition of the time interval
	$$
	\mathcal{T}_N \coloneqq \left\{ \{t_k\}_{k=0}^N \mid 0= t_0<t_1<\ldots<t_N\leq T \right\},
	$$
	and a time-discretization $\boldsymbol{h}_N$ of $h$
	$$
	\boldsymbol{h}_N\coloneqq\left\{ \{h_k\}_{k=1}^{N} \mid h_k \in \ell^p(X,\mu)\right\}.
	$$
	Having fixed $\epsilon >0$, we call $\mathcal{D}_\epsilon\coloneqq (\mathcal{T}_N,\boldsymbol{h}_N)$ an \emph{$\epsilon$-discretization} of $([0,T]; h)$ if
	\begin{itemize}
		\item $t_k-t_{k-1} \leq \epsilon$ for every $k=1,\ldots, N$ and $T- t_N \leq \epsilon$;
		\item $\sum_{k=1}^{N} \int_{t_{k-1}}^{t_{k}} \left\| h(t) - h_k \right\|_p dt \leq \epsilon$.
	\end{itemize}
\end{definition}

Given an $\epsilon$-discretization $\mathcal{D}_\epsilon$, we consider the following system of difference equations which arises from an implicit Euler-discretization of \eqref{model_problem3}:

\begin{equation}\label{implicit_Euler}
	(\operatorname{id} + \lambda_k \mathcal{A} )u_k= u_{k-1} + \lambda_k h_k, \qquad
	\lambda_k\coloneqq t_k - t_{k-1}\mbox{ and } k=1,\ldots,N,
\end{equation}
with $u_0\in\ell^p(X,\mu)$ given and $u_k \in \dom_p(\mathcal{A})$. 

\begin{definition}[$\epsilon$-approximate solution]\label{epsilon-approximation}
If the system \eqref{implicit_Euler} admits a solution $\boldsymbol{u}_\epsilon =\left\{u_k\right\}_{k=1}^N$ such that $u_k\in \dom_p(\mathcal{A})$ for every $k=1,\ldots, N$, then we define $u_\epsilon$ as the piecewise constant function
\begin{equation}\label{epsilon_approximation}
u_\epsilon(t)\coloneqq \begin{cases}
	\sum_{k=1}^N u_k\mathds{1}_{(t_{k-1},t_k]}(t) & \mbox{for } t\in (0, t_N],\\
	u_0 & \mbox{for } t=0,
\end{cases}
\end{equation}
and we call $u_\epsilon$ an \emph{$\epsilon$-approximate solution} of the Problem \eqref{model_problem3} (subordinate to $\mathcal{D}_\epsilon$).
\end{definition}

\begin{definition}[Mild solution]\label{def:weak_solution}
If $T<+\infty$, we say that $u \colon [0,T]\to \ell^p\left(X,\mu\right)$    is a \emph{mild solution} of Problem \eqref{model_problem3} if $u \in C\left( [0,T]; \ell^p\left(X,\mu\right)\right)$ and $u$ is obtained as a uniform limit of $\epsilon$-approximate solutions. Namely, for every $\epsilon>0$ there exists an $\epsilon$-discretization  $\mathcal{D}_\epsilon$ of $([0,T]; h)$, as in Definition \ref{def:epsilon-discretization}, and an $\epsilon$-approximate solution $u_\epsilon$  subordinate to $\mathcal{D}_\epsilon$, as in \eqref{epsilon_approximation}, such that
	\begin{equation*}
	\left\|u(t) - u_\epsilon(t) \right\|_p < \epsilon \qquad \mbox{for every } t\in [0,t_N]\subseteq [0,T].
	\end{equation*}
	
If $T = +\infty$, then  we say that $u$ is a mild solution of \eqref{model_problem3} if the restriction of $u$ to each compact subinterval $[0,a]\subset [0,+\infty)$ is a mild solution of \eqref{model_problem3} on $[0,a]$.
\end{definition}

Mild solutions as defined above are also known in the literature as $C^0$ solutions, see \cite[Chapter IV.8]{showalter2013monotone}. Every strong solution is a mild solution, see e.g., \cite[Chapter 4]{barbu2010nonlinear}, the converse is true under suitable assumptions, see Theorem \ref{thm:regularity} below.

\par

 It is noteworthy to recall that in the literature another notion of mild solution, given in terms of semigroup and inspired by the Duhamel's principle, is available (see, e.g., \cite[Definition 2.3, Chapter 4]{pazy}, \cite[Definition 3.1]{grillo2024blow}). We refer to \cite[Theorem 5.7]{benilan1988evolution} for a proof of the equivalence of the two notions in the linear setting.

\subsection{Existence, uniqueness and regularity of mild solutions}
In this section we state the main results about existence, uniqueness and regularity of mild solutions to Problem \eqref{model_problem3} with $\mathcal{A}$ and $F$ defined as in \eqref{A}, with $f: \R \rightarrow \R$ satisfying one of the following conditions:
\begin{enumerate}
	 \item[\textbf{\textnormal{\ref{F1:monotonicity}}}] $f$ is continuous and  monotone decreasing with $f(0)=0$;
	 \item[\textbf{\textnormal{\ref{F2:Lipschitz}}}] $f$ is uniformly Lipschitz with constant $L>0$ and $f(0)=0$.
	 \end{enumerate} 
The proofs of the results will be given in Section \ref{s:proofs}. We start by discussing the existence and the sign of mild solutions.

\begin{theorem}\label{thm:main2}
	Let $G$ be a connected graph and $1\leq p <\infty$, with the additional Condition~\eqref{eq:Delta(C_c)} if $p>1$. Furthermore, suppose that $f$ satisfies either Condition \textnormal{\ref{F1:monotonicity}} or Condition \textnormal{\ref{F2:Lipschitz}}. Let $u_0 \in \ell^p(X,\mu)$ and $h \in L^1_{\textnormal{loc}}\left([0,T]; \ell^p\left(X,\mu\right)\right)$.

	If one of the following additional conditions holds:  
	\begin{enumerate}[a)]
		\item\label{item:nonnegativity/nonpositivity} $u_0,h(t)\geq 0$ (or $\leq 0$) for all $t\geq 0$;
		\item\label{hpA} $u_0$ or $h(t)$ changes sign and $G$ satisfies one of the conditions $\textnormal{\ref{IP}}$ or $\textnormal{\ref{B}}$; 
	\end{enumerate}
	then there exists a unique mild solution $u$ of \eqref{model_problem3}. Under hypothesis \ref{item:nonnegativity/nonpositivity}, $u(t)\geq 0$ (or $\leq 0$) for every $t\geq 0$.
	
	Furthermore, for every $\epsilon >0$ there exists a continuous function $\delta \colon [0,\infty) \to [0,\infty)$ such that $\delta(0)=0$ and if $u_\epsilon$ is an $\epsilon$-approximate solution of \eqref{model_problem3}
\begin{equation}\label{uniform_limit}
		\| u(t) - u_\epsilon(t)\|_p\leq \delta(\epsilon) \qquad \mbox{for } t \in [0,T-\epsilon].
	\end{equation}	   
	Moreover, for any pair $(u_0,h),(\hat{u}_0,\hat{h})$, the corresponding mild solutions $u,\hat{u} \in C\left([0,T];  \ell^p\left(X,\mu\right)\right)$ satisfy
	\begin{equation}\label{contraction_of_solutions}
		\left\| u(t_2) - \hat{u}(t_2)  \right\|_p \leq e^{\omega(t_2-t_1)}\left\| u(t_1) - \hat{u}(t_1)  \right\|_p   + \int_{t_1}^{t_2}e^{\omega(t_2-s)}\left\| h(s) - \hat{h}(s) \right\|_p\, ds, \quad \forall \, 0\leq t_1<t_2 \leq T,
	\end{equation}
with $\omega=0$, resp. $\omega=L$, if condition \textnormal{\ref{F1:monotonicity}}, resp. \textnormal{\ref{F2:Lipschitz}}, hold.	
\end{theorem}
The proof of Theorem \ref{thm:main2} relies on the accretivity or $\omega$-accretivity of the operator $\mathcal A$ and the existence of an $\epsilon$-approximate solution which will be discussed in the following section.  \par
With regard to the regularity of mild solutions, the following statements hold. The proof of Theorem \ref{thm:regularity} is again based on the accretivity or $\omega=L-$accretivity for some $L>0$ of the operator $\mathcal A$, while the proof of Theorem \ref{thm:regularity2} relies on its continuity under the stated assumptions.

\begin{theorem}\label{thm:regularity}
Let the assumptions of Theorem \ref{thm:main2} hold with the restriction that $1< p <\infty$.  If $u_0 \in  {\dom_p(\mathcal{A})}$ and $h \in W^{1,1}\left([0,T]; \ell^p\left(X,\mu\right)\right)$, then the unique mild solution $u$ of \eqref{model_problem3} is a strong solution and satisfies $u\in W^{1,\infty}\left([0,T]; \ell^p\left(X,\mu\right)\right)$. Moreover, $u$ is everywhere differentiable from the right, its right derivative is right continuous and satisfies the estimate
$$\Big\|\frac{d u}{d t}(t)\Big\|_p\leq e^{\omega t}\|h(0)- \mathcal{A} u_0\|_p+\int_0^t e^{\omega(t-s)} \Big\|\frac{d h}{d s}(s)\Big\|_p\, ds \qquad \text{for } a.e.\, t \in (0,T)$$
with $\omega=0$, resp. $\omega=L$, if condition \textnormal{\ref{F1:monotonicity}}, resp. \textnormal{\ref{F2:Lipschitz}}, hold.
\end{theorem}

\begin{theorem}\label{thm:regularity2}
Let the assumptions of Theorem \ref{thm:main2} hold with the restriction that $f$ satisfies condition~\ref{F2:Lipschitz} and $\sup_{x\in X}\frac{\sum_{y\in X}w(x,y)+\kappa(x)}{\mu(x)}<\infty$. Then, the unique mild solution $u$ of \eqref{model_problem3} (with $u_0 \in \ell^p(X,\mu)$, $h \in L^1_{\textnormal{loc}}\left([0,T]; \ell^p\left(X,\mu\right)\right)$ and $1\leq p <\infty$) is a strong solution and 
$$
u(t)=u(0)-\int_0^t\mathcal Au(s)ds+\int_0^th(s)ds\qquad \forall t\in (0,T).
$$
Moreover, if $h\in C([0,T];\ell^p(X,\mu))$, then $u$ is a classical solution.
\end{theorem}

\section{A related time-independent equation}\label{auxiliary section}

Motivated by \eqref{implicit_Euler} and with the aim to construct mild solutions to \eqref{model_problem3}, in this section we study the time-independent equation:
\begin{equation}\label{model_equation(discrete)}
 (\operatorname{id} + \lambda \mathcal{A})u = g \quad \text{in }X,
\end{equation}
where $g \in \ell^p(X,\mu)$, $\lambda>0$ and $\mathcal{A}$ is as defined in \eqref{A} with $F$ satisfying condition \textnormal{\ref{F1:monotonicity}} or \textnormal{\ref{F2:Lipschitz}}. More precisely, we will prove the following statement.

\begin{theorem}\label{th: F1}
	Let $G$ be a connected graph and $1\leq p <\infty$. Furthermore, if $p>1$, assume that condition~\eqref{eq:Delta(C_c)} holds. Then, there exists a dense subset $\Omega_p \subseteq \dom_p(\mathcal{A})$ such that
\begin{itemize} 
\item 
if \textnormal{\ref{F1:monotonicity}} holds, then $\mathcal{A}_{|\Omega_p}$ is accretive on $\ell^p(X,\mu)$;
\item 
if \textnormal{\ref{F2:Lipschitz}} holds with constant $L>0$, then $\mathcal{\mathcal A}_{|\Omega_p}$ is $\omega$-accretive on $\ell^p(X,\mu)$, with $\omega = L$.
\end{itemize}

Moreover,
\begin{itemize}
\item[\textnormal{\textbf{(I)}}] if \textnormal{\ref{F1:monotonicity}} holds, for every $\lambda >0$ and $g \in \ell^{p,\pm}(X,\mu)$, then there exists a unique solution to equation \eqref{model_equation(discrete)} satisfying $u \in \ell^{p,\pm}(X,\mu)\cap\Omega_p$
and the estimate  
\begin{equation}\label{contractivity}
    \|u\|_p\leq C \|g\|_p,
\end{equation}
with $C=1$.

If \textnormal{\ref{F2:Lipschitz}} holds with constant $L>0$, then for every $\lambda \in (0,1/L)$ and $g \in \ell^{p,\pm}(X,\mu)$, then there exists a unique solution to equation \eqref{model_equation(discrete)} satisfying $u \in \ell^{p,\pm}(X,\mu)\cap\Omega_p$ 
and the estimate  
\begin{equation}\label{contractivitybis}
    \|u\|_p\leq C \|g\|_p,
\end{equation}
with $C=(1-\lambda L)^{-1}$.  

\item[\textnormal{\textbf{(II)}}] If $G$ satisfies one of the conditions $\textnormal{\ref{IP}}$ or $\textnormal{\ref{B}}$, then $\Omega_p=\dom_p(\mathcal{A})$ and the  statements in \textnormal{\textbf{(I)}} hold for every $g\in \ell^p(X,\mu)$, with $2^{(p-1)/p}C$ instead of $C$.
\end{itemize}
\end{theorem}
As mentioned in the Introduction, the proof of Theorem \ref{th: F1} relies on an approximation scheme defined by considering an increasing exhaustion of the graph $X$ by finite sets $\{X_n\}_n$. The idea is to consider a sequence of operators $\mathcal{A}_n$ and data $g_n$ defined on finite Dirichlet subgraphs with nodes sets $X_n$ and to solve, on each of them, the equation $(\operatorname{id} + \lambda \mathcal{A}_n)u_n = g_n$ in $X_n$. The solutions $\{u_n\}_n$ form a sequence that converges, monotonically in case $\textbf{(I)}$ or by dominated convergence in case $\textbf{(II)}$, to a solution of \eqref{model_equation(discrete)} which satisfies the required norm estimates.  The accretivity (or $\omega$- accretivity) of $\mathcal{A}_{|\Omega_p}$ on a dense subset ensures uniqueness.
A key role in the proof of the accretivity issue is played by the definition of the set $\Omega_p$, which consists of functions in $\text{dom}_p(A)$ that can be approximated by sequences supported on the finite sets $X_n$, see Definition \ref{def:Omega} below. Then, the accretivity of $\mathcal{A}$ is recovered by exploiting that of the approximating operators $\mathcal{A}_n$. The additional assumptions $\textnormal{\ref{IP}}$ or $\textnormal{\ref{B}}$ in part $\textbf{(II)}$ are technical and guarantee injectivity of $\text{id}+\lambda \mathcal{A}$ on the full domain, this ensures the convergence of the whole approximating sequence and enables us to replace $\Omega_p$ with $\text{dom}_p(\mathcal{A})$. It is worth noting that the exhaustion method adopted here seems the most natural for preserving the actual geometry of the problem. Indeed, the approximation is spatial and dictated by the combinatorial structure of the graph. In particular, this differs from the classical Galerkin method, which relies on arbitrary finite‑dimensional subspaces (typically spanned by a chosen basis) that do not necessarily reflect the geometry of the graph.

\begin{remark}\label{staz} By setting $\alpha = 1/\lambda$ and recalling the definition of the operator 
$\mathcal{A}$, Theorem~\ref{th: F1}-\textnormal{\textbf{(I)}} and \textnormal{\textbf{(II)}} (under the stated assumptions) yields the 
existence and uniqueness of solutions to the equation
\begin{equation*}
  \Delta u + \alpha\, u = f(u) + \alpha\, g
  \qquad \text{in } X,
\end{equation*}
for every $\alpha>0$ provided that condition~\textnormal{\ref{F1:monotonicity}} 
is satisfied (or for $\alpha > L$ when condition~\textnormal{\ref{F2:Lipschitz}} 
holds). If $\alpha = 0$, since $f(0)=0$, the equation admits the trivial 
solution $u=0$. When \textnormal{\ref{F1:monotonicity}} holds and $G$ does not contain any infinite path or satisfies condition $\textnormal{\ref{IP}}$, a straightforward modification of the proof of the comparison principle stated in Theorem \ref{thm:min_principle} below yields that $u=0$ is the unique solution in $\ell^p(X,\mu)$.
\end{remark}

The section is organized as follows: we begin with a comparison principle and positivity properties (Subsection \ref{comparison}), followed by preliminary results on finite graphs (Subsection \ref{finite}). In Subsection \ref{Omegap} we introduce the set $\Omega_p$ and we prove accretivity of $\mathcal{\mathcal A}_{|\Omega_p}$. Finally, Subsection \ref{proofTh41} combines the above ingredients to establish Theorem \ref{th: F1}.

\subsection{Comparison principle}\label{comparison}
We provide here a  comparison principle and positivity preserving properties for solutions of equation \eqref{model_equation(discrete)}.   See~\cite{bianchi2022generalized,bksw} for related results.

\begin{theorem}\label{thm:min_principle}
Let $G$ be a connected graph. Let $F$ be defined as in \eqref{A}, and suppose that $f$ satisfies \textnormal{\ref{F1:monotonicity}} and $\lambda>0$ or $f$ satisfies~\textnormal{\ref{F2:Lipschitz}} and $0<\lambda<1/L$. Let $u \in \dom\left(\Delta\right)$ and assume that one of the following cases holds:

\begin{itemize}
\item[\textbf{a)}] $G$ does not contain any infinite path;
\item[\textbf{b)}] $G$ satisfies $\textnormal{\ref{IP}}$ and $\exists\; p> 0$ such that for every infinite path $\{x_n\}_n$ we have  $\sum_{n}|u(x_{n})|^p\mu(x_{n})<\infty$. 
\end{itemize}

If $\left( {\operatorname{id}}+\lambda \mathcal A\right)u \geq 0$ $(\leq 0)$, then $u \geq 0$ $(\leq 0)$. Moreover, if $u$ vanishes at some node, then $u\equiv 0$.

\begin{proof}
Define $\psi(s):=s-\lambda f(s)$, $s\in\R$. If $f$ satisfies \textnormal{\ref{F1:monotonicity}} then $\psi$ is strictly monotone increasing and surjective for all $\lambda >0$,  while if $f$ satisfies \textnormal{\ref{F2:Lipschitz}} this holds for all $0<\lambda <1/L$. Indeed, if $0<\lambda <1/L$, then for any $s>s'$ one has $\psi(s)- \psi(s')= s-s' - \lambda [f(s)-f(s')]\geq (1-\lambda L)(s-s')>0$. Hence, $\psi$ is strictly increasing.  Furthermore, we have $|\psi(s)|=|s -\lambda f(s)|\geq (1-\lambda L)|s|$ for all $s\in \R$, therefore $\lim_{\pm \infty} \psi(s)=\pm \infty$ for all $0<\lambda <1/L$. Hence, $\psi$ is surjective in $\R$ as well. Let $u \in \dom\left(\Delta\right)$ be such that $\left( {\operatorname{id}}+\lambda \mathcal A\right)u=\left( \psi + \lambda\Delta\right)u \geq 0$. If $u\geq0$, then there is nothing to prove.
Hence, we assume that there exists $x_{0}\in X$ such that $u(x_{0}) <0$. We will show that this leads to a contradiction in the two cases.

Recalling that $f(0) = 0$ in both the cases \textnormal{\ref{F1:monotonicity}} and \textnormal{\ref{F2:Lipschitz}}, then $\psi(0)=0$ and, combined with its  strictly monotone increasing property, we get
\begin{equation}\label{eq:min_1}
	\psi(u(x_{0}))+ \lambda\frac{\kappa(x_{0})}{\mu(x_{0})}u(x_{0}) < 0.
\end{equation}
Furthermore, as $\left( \psi + \lambda\Delta\right)u(x_0) \geq 0$,
\begin{equation}\label{eq:min_2}
	0\leq  \psi(u(x_{0}))+\frac{\lambda}{\mu(x_{0})}\sum_{y\in X}w(x_{0},y)\left(u(x_{0})-u(y)\right) + \lambda\frac{\kappa(x_{0})}{\mu(x_{0})}u(x_{0}).
\end{equation}
Combining the above inequalities \eqref{eq:min_1} and \eqref{eq:min_2}, we get
\begin{equation*}
	0<-\left[ \psi(u(x_{0})) +\lambda\frac{\kappa(x_{0})}{\mu(x_{0})}u(x_{0}) \right]\leq \frac{\lambda}{\mu(x_{0})}\sum_{y\in X}w(x_{0},y)\left(u(x_{0})-u(y)\right),
\end{equation*}
and since $w(\cdot,\cdot)\geq 0$ and $G$ is connected, there exists $y=x_{1}\sim x_{0}$ such that $u(x_{1})< u(x_{0})$. In particular, $u(x_1)<0$.

Hence, we see that every node where $u$ is negative is connected to a node where $u$ is strictly smaller. This is the basic observation that will be used in the two cases.
\vspace{0.2cm}

\noindent \textbf{Case a)} Iterating the procedure above, we find a sequence of distinct nodes $\{x_{k}\}_{k=0}^n$ such that $x_{0}\sim x_{1}\sim \cdots \sim x_{n}$ and
$$
u(x_{n})< u(x_{n-1}) < \ldots < u(x_{0}) <0.
$$
Since $G$ does not have any infinite path this sequence must end which leads to a contradiction.

\vspace{0.2cm}

\noindent \textbf{Case b)} 
In this case, we can obtain an infinite path $\{x_{n}\}_n$ such that 
$$
\ldots < u(x_{n})< u(x_{n-1}) < \ldots < u(x_{0}) <0.
$$
It follows that $|u(x_{n})|^p >|u(x_{0})|^p >0$, for every $n$, and therefore $$\sum_{{n}}|u(x_{n})|^p\mu(x_{n})>|u(x_{0})|^p\sum_{n}\mu(x_{n})=\infty,$$ which gives a contradiction.
\vspace{0.5cm}

Hence, we have established that $u \geq 0$ in both cases. Now, if there exists $x_0 \in X$ such that $u(x_0)=0$, then
$$0\leq  \psi u(x_{0}) + \lambda\Delta u(x_{0}) =-\frac{\lambda}{\mu(x_{0})}\sum_{y\in X}w(x_{0},y)u(y) \leq 0 $$
and thus $u(y)=0$ for all $y\sim x_0$.
Using induction and the assumption that $G$ is connected we get $u\equiv 0$.

The proof that $\left( \psi + \lambda\Delta\right) u \leq 0$ implies $u \leq 0$ is completely analogous.
    
\end{proof}

\end{theorem}

\begin{corollary}\label{cor:min}
With the hypotheses of Theorem~\ref{thm:min_principle}, let $u_1, u_2$ be solutions of
$$
\left({\operatorname{id}} + \lambda\mathcal A\right) u_k = g_k, \qquad   g_k \in C(X) \mbox{ for } k=1,2.
$$
If $g_1\geq g_2$, then $u_1 \geq u_2$. In particular,  ${\operatorname{id}} +\lambda \mathcal{A}$ is injective on $\dom_p(\mathcal A)$ for all $1\leq p< \infty$.  
\end{corollary}
\begin{proof}
Let $\psi$ be defined as in the proof of Theorem \ref{thm:min_principle}; since it  is strictly increasing and $\psi(0)=0$, $\psi(a)-\psi(b)$ and $\psi(a-b)$ have the same sign for all $a,b\in\mathbb R$. Hence,    
there exists a positive function $\sigma\in C(X)$ such that
\[
\psi(u_1(x))-\psi(u_2(x))=\sigma(x) \psi(u_1(x)-u_2(x)),
\]
for all $x \in X$.
Indeed, we can define
\[
\sigma(x)=
\begin{cases}
1&\text{if } u_1(x)=u_2(x)\\
\frac{\psi(u_1(x))-\psi(u_2(x))}{\psi(u_1(x)-u_2(x))} &\text{otherwise}.
\end{cases}
\]
It follows that $u_1-u_2$ satisfies
\[
\sigma\psi(u_1-u_2) + \lambda \Delta(u_1-u_2)= g_1-g_2\geq 0 \,\,\text{ on } X.
\]
Then, the same argument used in the first part of proof of Theorem \ref{thm:min_principle} implies that $u_1-u_2\geq 0$.

    It follows that, if $u_k \in \dom_p(\mathcal A)$ for $k=1,2$ with $1\leq p< \infty$,  
    and
$$
g_1 =\left({\operatorname{id}} + \lambda\mathcal A\right) u_1 = \left({\operatorname{id}} + \lambda\mathcal A\right) u_2 = g_2
$$
then $u_1 = u_2$. In particular, ${\operatorname{id}} +\lambda \mathcal{A}$ is injective on $\dom_p(\mathcal A)$.  
\end{proof}

\subsection{Preliminary results on finite graphs}\label{finite}

\begin{lemma}\label{lem:accretivity_finite_graphp=1}
Let $G$ be a finite graph and $1\leq p<\infty$.
\begin{itemize}
\item[(i)] If     $f$ satisfies \textnormal{\ref{F1:monotonicity}}, then $\mathcal A$ is accretive on $\ell^p(X,\mu)$.
\item[(ii)] If     $f$ satisfies \textnormal{\ref{F2:Lipschitz}} with constant $L$, then $\mathcal A$ is $\omega$-accretive on $\ell^p(X,\mu)$, with $\omega = L$. 
\end{itemize}
 \end{lemma}
\begin{proof}
 Both statements follow from general theory. Statement $(ii)$ is proved in \cite[Example 2.2]{benilan1988evolution}. As for $(i)$, it can be obtained by combining \cite[Proposition 2.20]{benilan1988evolution} and \cite[Proposition 2.23]{benilan1988evolution}. For the reader’s convenience, we nevertheless include below a direct elementary proof well suited to our setting.

According to \eqref{accretivity2}, 
with 
\begin{equation}
		z \coloneqq (F+\Delta)u -  (F+\Delta)v,\qquad k \coloneqq u-v,
	\end{equation}
it suffices to show that 
\[
\sum_{x\in X} [\Delta(u-v)(x) +(Fu-Fv)(x)] |u(x)-v(x)|^{p-1}\operatorname{sgn}(u(x)-v(x))\mu(x)
\geq 0,
\]
for every $u,v\in \ell^p(X,\mu).$
We break up the above sum into two summands. The first is 
\[
\sum_{x\in X} \Delta(u-v)(x) |u(x)-v(x)|^{p-1}\operatorname{sgn}(u(x)-v(x))\mu(x),
\]
which is nonnegative by Lemma~\ref{lem:positivity}.
The second summand 
\[
\sum_{x\in X} (Fu(x)-Fv(x)) |u(x)-v(x)|^{p-1}\operatorname{sgn}(u(x)-v(x))\mu(x),
\]
is also nonnegative, since, 
by monotonicity of $f$ (recall that $F=-f$), we have
\begin{equation*}
 \operatorname{sgn}(u(x) - v(x)) = \operatorname{sgn}(Fu(x) - Fv(x)).
\end{equation*}

This concludes the proof of $(i)$.

\end{proof}

\begin{lemma}\label{lem:bijective-sign-norm}
Let $G$ be a finite graph and $1\leq p<\infty$. Then, for every $g\in C(X)$ equation \eqref{model_equation(discrete)}
admits a unique solution $u$ provided that
\begin{itemize}
    \item $\lambda > 0$ if $f$ satisfies condition~\textnormal{\ref{F1:monotonicity}}, or
    \item $0 < \lambda < 1/L$ if $f$ satisfies condition~\textnormal{\ref{F2:Lipschitz}} with constant $L>0$.
\end{itemize}
In addition, the solution $u$ preserves the sign of the datum $g$ and  the following a priori estimate holds:
\begin{equation}\label{apriori}
    \|u\|_p \leq C\|g\|_p \, ,\quad C = \begin{cases}
        1 &\mbox{if } f \mbox{ satisfies } \textnormal{\ref{F1:monotonicity}},\\
        (1-\lambda L)^{-1}  &\mbox{if } f \mbox{ satisfies } \textnormal{\ref{F2:Lipschitz}}.
    \end{cases}
\end{equation}
\end{lemma}
\begin{proof}
 Recalling that $\mathcal{A} = F + \Delta$, with $F = - f$,  we start noticing that equation \eqref{model_equation(discrete)} can be written as \begin{equation}\label{eqPsi}
(\psi+\lambda \Delta)u=g,
\end{equation}
where $\psi(s):=s-\lambda f(s)$. In the proof of Theorem \ref{thm:min_principle} we showed that if $f$ satisfies \textnormal{\ref{F1:monotonicity}} then $\psi$ is strictly monotone increasing and surjective for all $\lambda >0$,  while if $f$ satisfies \textnormal{\ref{F2:Lipschitz}} this holds for all $0<\lambda <1/L$.  

Then, since $\lambda\Delta$ is a diagonally dominant matrix, all assumptions of \cite[Theorem 1]{willson1968solutions} are satisfied and we conclude that, for every $g$, there exists a solution $u$ of \eqref{eqPsi} and, in turn, of~\eqref{model_equation(discrete)}. 

Still referring to equation \eqref{eqPsi}, the fact that $\operatorname{id}+\lambda\mathcal A$ is injective on $\dom_p(\mathcal A)$ for all $1\leq p< \infty$ and  preserves the sign of $g$ as soon as $\psi$ is strictly monotone increasing follows by case $a)$ of Theorem \ref{thm:min_principle} and Corollary \ref{cor:min}.

To prove \eqref{apriori}, let us start observing that 
\begin{align}
\| u\|_p^p &= \sum_{x\in X} u(x)\operatorname{sign}(u(x))|u(x)|^{p-1}\mu(x)\nonumber\\
&= \sum_{x\in X} g(x)|u(x)|^{p-1}\operatorname{sign}(u(x))\mu(x) - \lambda\sum_{x\in X} (Fu(x) + \Delta u(x))|u(x)|^{p-1}\operatorname{sign}(u(x))\mu(x)\nonumber\\
&\leq \sum_{x\in X} |g(x)||u(x)|^{p-1}\mu(x) - \lambda\sum_{x\in X} (Fu(x) + \Delta u(x))|u(x)|^{p-1}\operatorname{sign}(u(x))\mu(x) \nonumber\\
&\leq \|g\|_p \|u\|_p^{p-1} - \lambda\sum_{x\in X} (Fu(x) + \Delta u(x))|u(x)|^{p-1}\operatorname{sign}(u(x))\mu(x), \label{eq:Brezis_lemma}
\end{align}
where in the second equality we used the identity $u(x)=g(x) -\lambda(F + \Delta)u(x)$, and in the last inequality we applied Hölder's inequality. We now turn to deriving a lower bound for the second term in \eqref{eq:Brezis_lemma}. Since the inequality~\eqref{apriori} is trivially verified for $u = 0$, hereafter we will assume $u$ not identically zero. We have two cases:

\textit{Case of $f$ satisfying Condition \textnormal{\ref{F1:monotonicity}}:} Observe that 
\begin{equation*}
\sum_{x\in X} \Delta u(x)|u(x)|^{p-1}\operatorname{sign}(u(x))\mu(x) \geq 0 
\end{equation*}
by Lemma~\ref{lem:positivity}. 
From \ref{F1:monotonicity}, it holds that $\operatorname{sgn}(u(x)) = \operatorname{sgn}(Fu(x))$ and then
$$
\sum_{x\in X} Fu(x) |u(x)|^{p-1}\operatorname{sign}(u(x))\mu(x) = \sum_{x\in X} |Fu(x)||u(x)|^{p-1} \mu(x)  \geq 0.
$$
Combining the above results in~\eqref{eq:Brezis_lemma} we infer that 
$$
\|u\|_p^p \leq \|g\|_p \|u\|_p^{p-1},
$$
that is, $\|u\|_p \leq \|g\|_p$.

\noindent\textit{Case of $f$ satisfying Condition \textnormal{\ref{F2:Lipschitz}}:} We have that
\begin{align*}
\sum_{x\in X} (Fu(x) + \Delta u(x))|u(x)|^{p-1}\operatorname{sign}(u(x))\mu(x) \geq &-\sum_{x\in X} |Fu(x)||u(x)|^{p-1}\mu(x) \\
&-|F(0)| \sum_{x\in X} |u(x)|^{p-1}\mu(x)\\
&+ \sum_{x\in X} \Delta u(x)|u(x)|^{p-1}\operatorname{sign}(u(x))\mu(x)\\
\geq &-L\sum_{x\in X} |u(x)|^{p}\mu(x) -|F(0)| \sum_{x\in X} |u(x)|^{p-1}\mu(x)\\
\geq &-L\|u\|_p^p,
\end{align*}
where  in the second inequality we used Lemma~\ref{lem:positivity} and the Lipschitz assumption \ref{F2:Lipschitz}, and in the last inequality we used the fact that $F(0)=0$. Combining the above results in~\eqref{eq:Brezis_lemma} we infer that
\begin{equation*}
\|u\|_p^p \leq \|g\|_p \|u\|_p^{p-1} + \lambda L\|u\|_p^p,
\end{equation*}
and consequently 
\begin{equation*}
\|u\|_p \leq (1 - \lambda L)^{-1}\|g\|_p \quad \mbox{for every } 0<\lambda < 1/L.
\end{equation*}
\end{proof}

\subsection{The set $\Omega_p$}\label{Omegap}

Assume the validity of~\eqref{eq:Delta(C_c)} for a fixed $1\leq p < \infty$. We first introduce a sequence of operators $\Amin$ whose purpose is to approximate the operator $\mathcal{A}$ over an increasing exhaustion $\{X_n\}_n$ of $X$ by finite subsets, i.e., a sequence of finite subsets $X_n$ of $X$ such that $$X_n \subseteq X_{n+1} \quad \text{ and }\quad  X=\cup_{n=1}^{\infty} X_n\,.$$
\begin{definition}[The operator $\Amin$]\label{def:Amin}
For $n\geq 1$ let 
\begin{equation*}\label{eq:embedding-projection2}
\boldsymbol{\mathfrak{i}}_{n} \colon C(X_{n}) \hookrightarrow C(X), \quad  \boldsymbol{\pi}_n \colon C(X) \to C(X_n)
\end{equation*}
be the canonical embeddings and projections, respectively. 
Define 
	$$
	\Amin \colon \dom_p\left( \Amin \right)\subseteq \ell^{p}\left(X,\mu\right) \to \ell^{p}\left(X,\mu\right)
	$$
	by
	\begin{align*}
		\dom_p\left( \Amin \right)\coloneqq C_c(X), \quad \Amin u\coloneqq \boldsymbol{\mathfrak{i}}_{n}\left(F + \Deltadirn \right)\boldsymbol{\pi}_{n} u,
	\end{align*}
	where $\Deltadirn$ is the graph Laplacian associated to the Dirichlet subgraph $G_{\textnormal{dir},n} \subseteq G$ on the finite node set $X_n\subseteq X$.
\end{definition}

The sequence of operators $\Amin$ is involved in the definition of a subset $\Omega_p$ of $\dom_p(\mathcal{A})$. As we will see below, $\Omega_p$ is dense in $\dom_p(\mathcal{A})$ and $\mathcal{A}$ restricted to $\Omega_p$ is accretive.  
\begin{definition}[The set $\Omega_p$]\label{def:Omega}
 We define $\Omega_p \subseteq \dom_p(\mathcal{A})$ by letting
 \begin{equation*}\label{eq:omega2}
		\Omega_p\coloneqq \{u \in \dom_p(\mathcal{A}) \mid \exists\, \{u_n\}_n \mbox{ s.t. } \operatorname{supp}u_n\subseteq X_n,\; \lim_{n\to \infty}\|u_n -u \|_p=0,\; \lim_{n\to \infty}\|\Amin u_n -\mathcal{A}u \|_p=0 \}.
	\end{equation*}
\end{definition}
Let us remark that $\Omega_p = \dom_p(\mathcal{A}) = C(X)$
if $G$ is finite.

While the definition of $\Omega_p$ depends on the choice of the exhaustion,
this set always contains all finitely supported functions as will be shown in the next lemma.

\begin{lemma}\label{cor:FC}
We have that $C_c(X) \subseteq \Omega_p \subseteq \dom_p(\mathcal{A})$. In particular, $$\overline{\Omega_p}= \overline{\dom_p(\mathcal{A})}=\ell^p(X,\mu).$$	 
\end{lemma}
\begin{proof}
Let us fix $u \in C_c(X)$. Thanks to condition \eqref{eq:Delta(C_c)}, $u\in \dom_p(\mathcal{A})$ (recall Remark \ref{rem:Delta(C_c)_in_ell1}). Define 
	$$
	u_n(x)\coloneqq \boldsymbol{\mathfrak{i}}_{n}\boldsymbol{\pi}_n u(x)= \begin{cases}
		u(x) & \mbox{if } x \in X_n,\\
		0 &\mbox{otherwise}.
	\end{cases}
	$$
	Clearly, $\operatorname{supp}u_n \subseteq X_n$ and $\lim_{n\to \infty}\|u_n -u \|_p=0$. Since, $u\in C_c(X)$, there exists an $N>0$ such that $u(x)=0$ for every $x\in X\setminus X_N$.  
	
	We have  $u_n\in \dom\left(\Delta\right)\cap \left\{u \in C(X) \mid u \equiv 0 \mbox{ on } X\setminus X_n \right\}$
    for every $n\geq N$ and then by equation~\eqref{eq:Delta_dir_property_2}
	\begin{equation*}
		\Deltadirn \boldsymbol{\pi}_n  u_n (x) =  \Delta u_n (x) =  \Delta u (x) \quad \forall\; x\in X_n \text{ with } n\geq N,
	\end{equation*}
	that is,
	\begin{equation*}
		\Deltadirn \boldsymbol{\pi}_n  u_n = \boldsymbol{\pi}_n\Delta  u \quad \forall\; n\geq N.
	\end{equation*}
	Therefore, recalling that $F(0)=0$,
	\begin{align*}
		\Amin u_n &= \boldsymbol{\mathfrak{i}}_{n}\left( F + \Deltadirn \right) \boldsymbol{\pi}_n u_n\\
        &= \boldsymbol{\mathfrak{i}}_{n} F \boldsymbol{\pi}_n u_n  + \boldsymbol{\mathfrak{i}}_{n}\Deltadirn \boldsymbol{\pi}_n   u_n\\
        &= \boldsymbol{\mathfrak{i}}_{n}\boldsymbol{\pi}_n F  u_n  + \boldsymbol{\mathfrak{i}}_{n}\boldsymbol{\pi}_n\Delta  u \\
        &= \boldsymbol{\mathfrak{i}}_{n}\boldsymbol{\pi}_n\mathcal{A} u, \quad \forall\; n\geq N
	\end{align*}
	and, since $\mathcal{A}u \in \ell^p(X,\mu)$, it follows that $\lim_{n\to \infty}\| \Amin u_n - \mathcal{A} u\|_p=0$
	by dominated convergence.
\end{proof}

Finally, we prove the accretivity properties of $\mathcal{A}_{|\Omega_p}$.

\begin{lemma}\label{lem:L_Omega_accretive}
Let $1\leq p <  \infty$. 
\begin{itemize}
\item[(i)] If     $f$ satisfies \textnormal{\ref{F1:monotonicity}}, then $\mathcal{A}_{|\Omega_p}$ is accretive on $\ell^p(X,\mu)$.
\item[(ii)] If     $f$ satisfies \textnormal{\ref{F2:Lipschitz}} with constant $L$, then $\mathcal{A}_{|\Omega_p}$ is $\omega$-accretive on $\ell^p(X,\mu)$, with $\omega = L$. 
\end{itemize}
\end{lemma}
\begin{proof}
If $G$ is finite, then $\Omega_p =  \dom_p(\mathcal{A}) = C(X)$ and the conclusion holds by Lemma~\ref{lem:accretivity_finite_graphp=1}. If $G$ is infinite, let $u,v \in \Omega_p$. Then, by the definition of $\Omega_p$, there exists $\{u_n\}_n, \{v_n\}_n$ such that 
$$
\lim_{n\to \infty}\|u_n-u\|_p=\lim_{n\to \infty}\|v_n-v\|_p=0, \qquad \lim_{n\to \infty}\|\Amin u_n-\mathcal{A}u\|_p=\lim_{n\to \infty}\|\Amin v_n-\mathcal{A}v\|_p=0.
$$
By Lemma~\ref{lem:accretivity_finite_graphp=1} it follows that $\Amin$ is accretive in case $(i)$, resp. $\omega$-accretive with $\omega = L$ in case $(ii)$, on $\ell^p(X,\mu)$ for every $n$ (see also \cite[Corollary B.4]{bianchi2022generalized}) and, in turn, we deduce  that
\begin{align*}
	\left\|(u - v) + \lambda \left(\mathcal{A}u - \mathcal{A}v\right) \right\|_p &= \lim_{n\to \infty} \left\|(u_n - v_n) + \lambda \left(\Amin u_n - \Amin v_n\right) \right\|_p\\
    &\geq \begin{cases}
	   \lim_{n\to \infty}\|u_n - v_n\|_p= \|u - v\|_p & \mbox{in case } (i),\\
       \lim_{n\to \infty} (1- L\lambda)\|u_n - v_n\|_p= (1- L\lambda)\|u - v\|_p & \mbox{in case } (ii).
	\end{cases} 
\end{align*}
This completes the proof.
\end{proof}

\subsection{Proof of Theorem \ref{th: F1}}\label{proofTh41}

We have now all the ingredients to prove Theorem \ref{th: F1}. 

\begin{proof}[Proof of Theorem \ref{th: F1}]
 If $G$ is finite, then $\Omega_p= \dom_p(\mathcal{A})=C(X)$ and the assertions  in \textnormal{\textbf{(I)}} and \textnormal{\textbf{(II)}} follow from Lemmas~\ref{lem:accretivity_finite_graphp=1} and~\ref{lem:bijective-sign-norm}.

 From now on, let $G$ be infinite.  

We consider an increasing exhaustion of $X$ by finite connected subsets $X_n$ satisfying the additional condition 
\begin{equation}\label{eq:property_sets2}
	\{ x \in X_n \mid x\sim y \mbox{ for some } y \in X_{n+1}\setminus X_n \} \ne \emptyset,
\end{equation}
(see \cite[Lemma A.4 ]{bianchi2022generalized} where the existence of such an exhaustion was proved), and let $\Omega_p$ be the corresponding set defined in Definition~\ref{def:Omega}.

According to Lemma~\ref{lem:L_Omega_accretive}, $\Omega_p$  is dense in $\dom_p(\mathcal A)$ and in $\ell^p(X,\mu)$ and  $\mathcal{A}_{|\Omega_p}$  is either accretive if \textnormal{\ref{F1:monotonicity}} holds or $\omega$-accretive with constant $\omega=L$ if \textnormal{\ref{F2:Lipschitz}} holds.

To prove \textnormal{\textbf{(I)}}, let  $g\in \ell^{p,\pm}(X,\mu)$ be given. We want to prove that  for every $\lambda>0$, respectively, $0< \lambda<1/L$, there exists a unique $u\in \Omega_p$, with $u\geq 0$ if  $g\in \ell^{p,+}(X,\mu)$ or $u\leq 0$ if $g\in \ell^{p,-}(X,\mu)$, satisfying
\begin{equation}\label{u-pm-1}
(i)\quad (\operatorname{id} +\lambda \mathcal{A})u=g \qquad\text{ and } \qquad (ii)\quad \|u\|_p\leq C\|g\|_p
\end{equation}
with $C=1$ or $C=(1-\lambda L)^{-1}$ depending on whether \textnormal{\ref{F1:monotonicity}}
or \textnormal{\ref{F2:Lipschitz}} holds.

We assume first that $g\in \ell^{p,+}(X,\mu)$.

For each $n$, we define the subgraph 
\begin{equation}\label{eq:def_Dir_subgraph}
G_{\textnormal{dir},n} = (X_n, w_n, \kappa_{\textnormal{dir},n}, \mu_{n})\subset G
\end{equation} 
as a Dirichlet subgraph of $G$ (see Definition \ref{def:dir_subgraph}).
This amounts to defining
\begin{itemize}
	\item $w_n \equiv w_{|X_{n}\times X_{n}}$;\
	\item $\mu_n \equiv  \mu_{|X_{n}}$;
	\item for every $x \in X_n$, $b_{\textnormal{dir},n}(x)= \sum_{y \in \mathbullet{\partial}X_n}w(x,y)$;
	\item for every $x \in X_n$, $\kappa_{\textnormal{dir},n}(x)  = \kappa_{|X_n}(x) + b_{\textnormal{dir},n}(x)$.
\end{itemize}
If we define
\begin{align*}
&\mathbullet{\partial}X_{n,n+1}\coloneqq\{ y \in X_{n+1} \setminus X_n \mid \exists x \in X_n \mbox{ such that } x \sim y  \},
\end{align*}
which is not empty by \eqref{eq:property_sets2}, then, for every $x\in X_n$,  
\begin{align*}
\kappa_{\textnormal{dir},n}(x)  = \kappa_{|X_n}(x) + b_{\textnormal{dir},n}(x)
= \kappa_{\textnormal{dir},n+1}(x) +  \sum_{y \in \mathbullet{\partial}X_{n,n+1}}w(x,y).
\end{align*}
So, the collection $\{G_{\textnormal{dir},n}\}_{n\in \N}$ is a sequence of connected finite Dirichlet subgraphs such that each subgraph $G_{\textnormal{dir},n}$ is a Dirichlet subgraph of $G_{\textnormal{dir},n+1}$, that is, $$G_{\textnormal{dir},1}\subset \ldots \subset G_{\textnormal{dir},n}\subset G_{\textnormal{dir},n+1} \subset \ldots \subset G.$$

Denoting by 
\begin{equation*}\label{eq:embedding-projection2}
\boldsymbol{\mathfrak{i}}_{n,n+1} \colon C(X_n) \hookrightarrow C(X_{n+1}), \quad \boldsymbol{\mathfrak{i}}_{n} \colon C(X_{n}) \hookrightarrow C(X), \quad  \boldsymbol{\pi}_n \colon C(X) \to C(X_n)
\end{equation*}
the canonical embeddings and projections, respectively, define
$$
g_n \coloneqq  \boldsymbol{\pi}_{n}g \,.
$$
As in the proof of Lemma \ref{lem:bijective-sign-norm}, we set $\psi(s)=s-\lambda f(s)$. If $f$ satisfies \textnormal{\ref{F1:monotonicity}}, then $\psi$ is strictly monotone increasing and surjective for every $\lambda >0$,  while if $f$ satisfies \textnormal{\ref{F2:Lipschitz}} the same holds for all $0<\lambda <1/L$. For these values of $\lambda$, from the proof of Lemma \ref{lem:bijective-sign-norm} we know that for every $n\in \N$ there exist $\hat{u}_n \in C(X_n)_n$ such that $\hat{u}_{n} \geq 0$  and  
\begin{equation}\label{def:hatv_n}
(\psi + \lambda \Delta_{\textnormal{dir}, n})\hat{u}_{n}  = g_n.
\end{equation}
Setting
$$
q_n(x)\coloneqq ( \psi +\lambda\Delta_{\textnormal{dir},n+1}) \boldsymbol{\mathfrak{i}}_{n,n+1}\hat{u}_{n}(x) \in C(X_{n+1})\,,
$$
by the fact that every $G_{\textnormal{dir},n}$ is a Dirichlet subgraph of $G_{\textnormal{dir}, n+1}$ (see also \cite[Lemma A1]{bianchi2022generalized} for more details), we have
$$
(\psi + \lambda \Delta_{\textnormal{dir}, n+1})\boldsymbol{\mathfrak{i}}_{n,n+1}\hat{u}_{n}(x) =(\psi + \lambda \Delta_{\textnormal{dir}, n})\hat{u}_{n}(x)  \quad \forall x \in X_n,
$$
and
\begin{equation*}
q_n(x)=\begin{cases}
g_n(x) & \mbox{if } x \in X_n,\\
-\frac{\lambda}{\mu(x)}\sum_{y \in X_n}w(x,y)\hat{u}_{n}(y)& \mbox{if } x \in \mathbullet{\partial}X_{n,n+1},\\
0 & \mbox{if } x \in X_{n+1}\setminus (X_n \cup \mathbullet{\partial}X_{n,n+1}).
\end{cases}
\end{equation*}
Since $0\leq \hat{u}_n$ and  $0\leq g_{n+1}$, it follows that $q_n(x) \leq g_{n+1}(x)$ for every $x \in X_{n+1}\setminus X_n$. In particular, from the fact that $g_{n+1}(x)=g_n(x)$ for every $x \in X_n$, we get $q_n\leq g_{n+1}$ on $X_{n+1}$. Therefore,
\begin{equation*}
(\psi + \lambda \Delta_{\textnormal{dir}, n+1})\boldsymbol{\mathfrak{i}}_{n,n+1}\hat{u}_{n}(x) \leq (\psi + \lambda \Delta_{\textnormal{dir}, n+1})\hat{u}_{n+1}(x)  \quad \forall x \in X_{n+1}
\end{equation*}
and by Theorem \ref{thm:min_principle}$ a)$, we have $\boldsymbol{\mathfrak{i}}_{n,n+1}\hat{u}_{n}\leq \hat{u}_{n+1}$.

Furthermore, denoting by $\|\cdot \|_{p,n}$ the restriction of $\|\cdot\|_p$ to $C(X_n)$, if $\lambda$ satisfies the assumptions of Lemma \ref{lem:bijective-sign-norm},  by \eqref{apriori} 
\begin{equation}\label{eq:unifboundu_n}
0\leq  \hat{u}_{n}(x) \mu(x)^{1/p}\leq 
\|\hat{u}_n\|_{p,n} \leq C\|g_n\|_{p,n} \leq C\|g\|_p,
\end{equation}
with $C$  defined in Lemma \ref{lem:bijective-sign-norm}. Hence, for every fixed $x$, $\hat{u}_{n}(x)$ is bounded uniformly in $n$. We now define $$u_n:= {\mathfrak{i}}_{n}\hat{u}_{n}\in C_c(X).$$

Thus, for every fixed $x \in X$, the sequence $\left\{  u_{n}(x)\right\}$ is monotonic and bounded. We can then define   
$$
u(x) \coloneqq \lim_{n\to \infty}  u_{n}(x) \quad \mbox{for } x\in X\label{eq:limit1}.  
$$
By monotone convergence and \eqref{eq:unifboundu_n} it follows that 
\[
\lim_{n\rightarrow{\infty}}\|u_n-u\|_p=0 \,\,\text{ and } 
\|u\|_p\leq C\|g\|_p.
\]
To conclude we need to show that
\begin{equation}
\label{u-prop-2}
 u\in  \Omega_p \qquad \text{ and } \qquad   (id+\lambda \mathcal{A})u=g.
\end{equation}
Now, by construction, for every $n$, $\operatorname{supp}u_n\subseteq X_n$ and  
\begin{align*}
(\boldsymbol{\mathfrak{i}}_{n}\operatorname{id}_n\boldsymbol{\pi}_{n}   + \lambda \mathcal A_n)     u_n =  \boldsymbol{\mathfrak{i}}_{n}(\operatorname{id}_n+\lambda F+\lambda \Delta_{dir,n} ) \boldsymbol{\pi}_n u_{n}=
	   \boldsymbol{\mathfrak{i}}_{n}  g_n,
\end{align*}
and, since $g\in \ell^p(X,\mu)$, $\|\mathfrak{i}_{n} g_n- g\|_p\to 0$, so that
\[
\|(\boldsymbol{\mathfrak{i}}_{n}\operatorname{id}_n\boldsymbol{\pi}_{n}   + \lambda \mathcal A_n)     u_n - g\|_p\to 0.
\]
Next,  $u_n \in C_c(X) \subseteq \dom\left(\Delta\right)$ for every $n$, so $\Delta u_{n}$ is well-defined. Furthermore, for every $x \in X$,
\begin{equation}\label{Monotone/Dominate_convergence}
  \sum_{y\in X}w(x,y) u(y)=
\sum_{y\in X} \lim_{n\to \infty}w(x,y) u_{n}(y)= \lim_{n\to \infty}\sum_{y\in X}w(x,y) u_{n}(y)  
\end{equation}
by monotone convergence. Then, recalling that every $G_{\textnormal{dir},n}$ is a Dirichlet subgraph of $G$ for every $n\in \N$, we get 
$$
\begin{aligned}
&\phantom{\,\,=\,\,\,}u(x)+\lambda Fu(x)+\frac{\lambda}{\mu(x)} \left( \deg(x)u(x)-\sum_{y\in X} w(x,y)u(y) \right)\\&=
\lim_{n\to \infty} \left( u_n(x)+\lambda Fu_n(x)+\frac{\lambda}{\mu(x)}\left( \deg(x)u_n(x) - \sum_{y\in X}w(x,y)u_n(y) \right) \right)\\
&=\lim_{n\to \infty}  g_n (x)= g(x),
\end{aligned}
$$
where we used the continuity of $F$.
Since $u(x)$ and $g(x)$ are finite, from the above chain of equalities, we deduce   that 
$
\sum_{y\in X} w(x,y)u(y) <+\infty
$ 
and therefore $u$ belongs to the formal domain of $\Delta$,  and therefore of $\mathcal{A}$, and 
\[
(F + \Delta) u = \lambda^{-1}(g- u).
\]
It follows that $(F+\Delta) u\in \ell^p(X,\mu)$ so that $u\in \dom_p(\mathcal{A})$ and satisfies 
\[
(\operatorname{id} + \lambda \mathcal{A})u=g.
\]
Finally, 
\begin{align}\label{eq:thm_existence_Alb}
 	\|\mathcal A_n u_n - \mathcal{A}u\|_p & \leq \frac{1}{\lambda} \left( \|(\boldsymbol{\mathfrak{i}}_{n}\operatorname{id}_n\boldsymbol{\pi}_n + \lambda\mathcal A_n)u_n - g\|_p +  \| \boldsymbol{\mathfrak{i}}_{n}\operatorname{id}_n\boldsymbol{\pi}_n u_n - u \|_p \right) \nonumber \\
 	&= \frac{1}{\lambda} \left( \|(\boldsymbol{\mathfrak{i}}_{n}\operatorname{id}_n\boldsymbol{\pi}_n + \lambda\mathcal A_n)u_n - g\|_p +  \|  u_n - u \|_p \right)\to 0,
 \end{align}
as required to show that $u\in \Omega_p$. The uniqueness statement now follows from the established accretivity/$\omega$-accretivity of $\mathcal{A}_{|\Omega_p}$, while the nonnegativity of $u$ follows by the nonnegativity of $u_n$.

 If $g\in \ell^{p,-}(X,\mu)$, then the proof arguments are entirely symmetric and the nonpositive solution $u$ can be built as a monotone decreasing limit. This concludes the proof of \textnormal{\textbf{(I)}}.

Now we prove \textnormal{\textbf{(II)}}. Given $g\in \ell^p(X,\mu)$, we let
\begin{equation}\label{pm}
	g_n \coloneqq \boldsymbol{\pi}_{n}g,\quad
	g^+_n\coloneqq \max\{0, \; g_n \},\quad
	g^-_n\coloneqq \min\{0, \; g_n \}.
\end{equation} 
From Lemma \ref{lem:bijective-sign-norm}, there exist $\hat{u}_{n}, \hat{u}^+_{n}, \hat{u}^-_{n}\in C(X_n)$ that satisfy
\begin{equation}\label{eq:Laccretive}
	\begin{cases}
		(\psi + \lambda\Delta_{\textnormal{dir},n})\hat{u}_{n} = g_n,\\
		(\psi +\lambda\Delta_{\textnormal{dir},n})\hat{u}^+_{n}  = g^+_n,\\
		(\psi +\lambda\Delta_{\textnormal{dir},n})\hat{u}^-_{n}  = g^-_n.
	\end{cases}
\end{equation}
Thus, having set
\begin{equation}\label{eq:def_u_n}
 u_n \coloneqq \mathbf{\mathfrak{i}}_{n}\hat{u}_n\quad  u_n^+ \coloneqq \mathbf{\mathfrak{i}}_{n}\hat{u}_n^+\quad  u_n^- \coloneqq \mathbf{\mathfrak{i}}_{n}\hat{u}_n^-,
\end{equation}
by construction
\begin{align*}
(\boldsymbol{\mathfrak{i}}_{n}\operatorname{id}_n\boldsymbol{\pi}_{n}   + \lambda \mathcal A_n )u_n 
	&= \boldsymbol{\mathfrak{i}}_{n}  g_n
\end{align*}
and therefore
$$
\lim_{n\to \infty}\|\left(\boldsymbol{\mathfrak{i}}_{n}\operatorname{id}_n\boldsymbol{\pi}_{n} + \lambda \mathcal A_n\right) u_n- g\|=0.
$$
Note that, by  
Corollary \ref{cor:min} and monotone limits, it holds that
\begin{align*}
	& u_n(x) = \hat{u}_n(x)\leq  \hat{u}^+_n(x)\leq u^+(x) \quad \mbox{if } x\in X_n,\\
	& u_n(x) = 0\leq u^+(x) \quad \mbox{if } x\notin X_n,
\end{align*}
and
\begin{align*}
	&  u^-(x)\leq \hat{u}^-_n(x)\leq \hat{u}_n(x)=u_n(x)    \quad \mbox{if } x\in X_n,\\
	& u^-(x)\leq 0= u_n(x) \quad \mbox{if } x\notin X_n.
\end{align*}
In particular, 
$$
u^-(x)\leq  u_n(x)\leq  u^+(x), \quad \forall\, x\in X,\;\forall\, n\in \N.
$$
Since and $u^\pm$ satisfy the conclusions of part \textnormal{\textbf{(I)}} with $g^{\pm}$ instead of $g$,  
  for every fixed $x\in X$ and every $n$, 
\begin{equation}\label{eq:uniform_bound}
	|u_n(x)| \leq 
    \max\left\{|u^-(x)|, |u^+(x)|\right\}<\infty, 
\end{equation}
and, for every $x$, the sequence $\{u_n(x)\}$ is  
bounded. 

Thus, by  a diagonal argument, there exists a  subsequence $\{u_{n'}\}$  such that 
\begin{align*}
	u(x)&\coloneqq \lim_{n'\to \infty} u_{n'}(x)\qquad x\in X
\end{align*}
 is well-defined and satisfies 
\begin{equation}\label{eq:uniform_bound2}
	|u(x)| \leq 
    \max\left\{|u^-(x)|, |u^+(x)|\right\}<\infty. 
\end{equation}
Thus,  
\[
\|u\|_p\leq \|u^+\|_p+\|u^-\|_p\leq 
C(\|g^+\|_p+\|g^-\|_p)
\leq 2^{(p-1)/p}C \|g\|_p,
\]
with $C$ as  in Lemma \ref{lem:bijective-sign-norm}, where the last  inequality follows from 
\[
(\|g^+\|_p + \|g^-\|_p)^p\leq 2^{p-1} (
\|g^+\|_p^p + \|g^-\|^p_p) = 2^{p-1} \|g\|^p_p.
\]
Moreover, since  $u^+,u^- $ 
 belong to $\dom_p(\mathcal{A})$ and 
\[
|u_{n'}(x)|\leq |u^+(x)|+|u^-(x)| \quad \quad \forall x \in X
\]
 by dominated convergence we have
\[
\sum_{y\in X} w(x,y)|u(y)|=\lim_{n'}\sum_{y\in X} w(x,y)|u_{n'}(y)|
\leq \sum_{y\in X} w(x,y) (|u^+(y)|+|u^-(y)|)<+\infty,
\] 
showing that $u$ belongs to $\dom(\Delta)$, and then, 
 arguing as in case \textnormal{\textbf{(I)}} (using dominated convergence instead of monotone convergence
 ),
 we deduce that 
 \begin{equation}
 \begin{split}
\label{eq:properties-$u_{n'}$}
& (i) \,\, \|u-u_{n'}\|_p\to 0 \qquad (ii)\,\, u\in \dom_p(\mathcal{A}) \\
& (iii) \,\, (\operatorname{id} +\lambda \mathcal{A})u =g  \qquad (iv)\,\,  \|\mathcal{A}u- \mathcal{A}_{n'} u_{n'}\|_p\to 0.
\end{split}
\end{equation}
\par
We note explicitly that   $(ii)$--$(iv)$ follow from $(i)$ and from the properties of $\{u_n\}$.
In particular, the above argument shows that for every fixed 
$\lambda>0$, respectively  $\lambda\in(0,1/L)$, 
depending on whether \textnormal{\ref{F1:monotonicity}}  or  \textnormal{\ref{F2:Lipschitz}} holds, and for every $g\in\ell^p(X,\mu)$, there exists a sequence $u_n\in C_c(X)$ with $\operatorname{supp} u_n\subseteq X_n$ such that, from every subsequence $\{u_{n'}\}$ of $\{u_n\}$ there exist a  sub-subsequence  $\{u_{n''}\}$ and a function $u$ for which
\eqref{eq:properties-$u_{n'}$} (i)--(iv) hold.

In order to complete the proof we need to show that the limit function $u$ is independent of the sub-subsequence $\{u_{n''}\}$ so that the original sequence $\{u_n\}$ converges to $u$ and therefore satisfies \eqref{eq:properties-$u_{n'}$} (i)--(iv).

\vspace{0.2cm}
\noindent\textbf{Case \textnormal{\ref{IP}}}.

According to  Corollary \ref{cor:min}, $\operatorname{id} +\lambda\mathcal A$ is injective on $\dom_p(\mathcal{A})$. This implies that all sub-subsequences of $\{u_n\}$ must converge to the same limit $u$. Thus the whole sequence $\{u_n\}$ converges to $u\in \dom_p(\mathcal{A})$ and  
\[
\|u_{n}-u\|_p\to 0, \qquad
\|\mathcal{A}u- \mathcal{A}_nu_n\|_p\to 0,
\]
so that $u\in \Omega_p$ by definition. Finally, since $\operatorname{id} + \lambda \mathcal{A}$ is 
surjective from $\Omega_p$ onto $\ell^p(X,\mu)$ and injective on its domain, for $\lambda$ in the appropriate range,  we conclude that $\Omega_p=\dom_p(\mathcal{A})$.
\vspace{0.2cm}

\vspace{0.2cm}
\noindent\textbf{Case \textnormal{\ref{B}}.} %

Let us write $\Delta_0$ for the graph Laplacian without killing term. Since \textnormal{\ref{B}} holds, $\Delta_0$ is bounded on $\ell^p(X,\mu)$ for every $1\leq p< \infty,$ and 
$\dom_p(\mathcal{A})=\{u\in \ell^p(X,\mu)\,:\, (F + \kappa/\mu) u\in \ell^p(X,\mu)\}$. 
Thus, if $\{X_n\}$ is any exhaustion of $X$ by finite sets, given $u\in \ell^p(X,\mu)$, $1\leq p<\infty$, and having set $u_n=u\mathbf{1}_{X_n}$
we have 
\[
\|u-u_n\|_p\to 0, \qquad \|\Delta_0 u-\Delta_0 u_n\|_p\to 0.
\]
Thus, applying Lemma~\ref{lem:positivity} yields
\[
\sum_{x\in X}\Delta_0 u(x)|u(x)|^{p-1}(\operatorname{sgn} u(x)) \mu(x) = \lim_n 
\sum_{x\in X}\Delta_0 u_n(x)|u_n(x)|^{p-1}(\operatorname{sgn} u_n(x)) \mu(x)\geq 0, 
\]
and therefore $\Delta_0$ is accretive on $\ell^p(X,\mu)$ and $\Delta$ is accretive on its domain.

Now, if $f$ satisfies condition \textnormal{\ref{F1:monotonicity}}, the monotonicity of $f$, arguing as in the proof of Lemma \ref{lem:accretivity_finite_graphp=1}, shows that, for every $u,v\in \dom_p(\mathcal{A})$
\[
\sum_{x\in X} [(F + \kappa(x)/\mu(x))u(x)-(F + \kappa(x)/\mu(x))v(x)] |u(x)-v(x)|^{p-1}\operatorname{sgn}(u(x)-v(x))\mu(x) \geq 0,
\]
and therefore $\mathcal{A}$ is accretive on its domain, hence $\operatorname{id}+\lambda\mathcal A$ is injective for all $\lambda>0$. 

If condition \textnormal{\ref{F2:Lipschitz}} holds, then $F$ maps $\ell^p(X,\mu)$ into itself, so that $\dom_p(\mathcal{A})=\dom_p(\Delta)$. Since $s\mapsto  Ls-f(s)$ is monotonically increasing, the above computation shows that 
\[
\sum_{x\in X} [(Fu(x)-Fv(x))+ L(u(x)-v(x) )]|u(x)-v(x)|^{p-1}\operatorname{sgn}(u(x)-v(x))\geq 0,
\]
and therefore $\mathcal{A}$ is $\omega=L$-accretive on $\dom_p(\mathcal{A})=\dom_p(\Delta)$.
According to \eqref{eq:omega-accretivity-1}
\[
\|(\operatorname{id}+\lambda\mathcal{A})u-(\operatorname{id}+\lambda\mathcal{A})v
\|_p\geq (1-L\lambda)\|u-v\|_p,
\]
for every $u,v\in\ell^p(X,\mu)$
and we deduce that $\operatorname{id}+\lambda \mathcal{A}$
is injective for $\lambda<L^{-1}$.

In both cases the proof is concluded as in  \textnormal{\ref{IP}}.
\end{proof}

\section{Proofs of the main results}\label{s:proofs}

 \textbf{Proof of Theorem \ref{thm:main2}.} Let $1\leq p<\infty$. From Theorem \ref{th: F1} there exists a dense subset $\Omega_p \subseteq \dom_p(\mathcal{A})$ such that the operator $\mathcal{\mathcal A}_{|\Omega_p}$ is either accretive or $L-$accretive for some $L>0$ and such that for every $\epsilon>0$ there exists an $\epsilon$-approximate solution as in Definition \ref{epsilon-approximation}. Then, the existence and uniqueness of mild solutions for \ref{model_problem} for every $u_0\in \overline{\Omega}_p=\ell^p(X,\mu)$ (as shown in Lemma \ref{cor:FC}) readily follows from \cite[Theorem 4.1]{barbu2010nonlinear} which is an extension of the Crandall-Liggett theorem, see \cite{crandall1971generation}.  $\Box$

 \textbf{Proof of Theorem \ref{thm:regularity}.}
In view of the $\omega$- accretivity of $\mathcal A$ (with $\omega=0$ or $\omega=L$) shown in Theorem \ref{th: F1}, the statement of Theorem \ref{thm:regularity} is a straightforward consequence of \cite[Theorems 4.5 and 4.6]{barbu2010nonlinear}, see also \cite[Remark 4.1]{barbu2010nonlinear}. We observe that the theorems just mentioned hold in reflexive and uniformly convex spaces, hence the assumption $p>1$.
$\Box$

 \textbf{Proof of Theorem \ref{thm:regularity2}.}
The proof of Theorem \ref{thm:regularity2} relies on the application of the following general result (given, e.g., in \cite[Theorem 1.6]{benilan1988evolution}). 

\begin{proposition}\label{Prop-strongclassical}
Let $h\in L^1_{\rm loc}([0,T];\ell^p(X,\mu))$, $1\leq p<\infty$, ${\rm{dom}}_p(\mathcal A)$ be closed and $\mathcal A$ continuous on its domain. If $u$ is a mild solution on $(0,T)$ of \eqref{model_problem3}, then $u$ is a strong solution and 
$$
u(t)=u(0)-\int_0^t\mathcal Au(s)ds+\int_0^th(s)ds\qquad \forall t\in (0,T).
$$
Moreover, if $h\in C([0,T];\ell^p(X,\mu))$, then $u$ is a classical solution. 
\end{proposition}
Proposition \ref{Prop-strongclassical} applies in the case when $f$ satisfies condition~\ref{F2:Lipschitz} and $\sup_{x\in X}\frac{\sum_{y\in X}w(x,y)+\kappa(x)}{\mu(x)}<\infty$. Indeed, it is known \cite[Theorem 2.15]{keller2021graphs} that if $\sup_{x\in X}\frac{\sum_{y\in X}w(x,y)+\kappa(x)}{\mu(x)}<\infty$, then $\Delta$ is continuous on $\ell^p(X,\mu)$. Moreover, by condition~\ref{F2:Lipschitz} if $u\in \ell^p(X,\mu)$, then $\sum_{x\in X}|F(u(x))|^p\mu(x)\leq L^p \sum_{x\in X}|u(x)|^p\mu(x)$ and if $\{u_n\}$ is convergent to $u$ in $\ell^p(X,\mu)$, then 
$$
\sum_{x\in X}|F(u_n(x))-F(u(x))|^p\mu(x)\leq L^p \sum_{x\in X}|u_n(x)-u(x)|^p\mu(x).
$$
Hence ${\rm{dom}}_p(\mathcal A)=\ell^p(X,\mu)$ and $\mathcal A$ is continuous.  Then, the conclusions of Proposition \ref{Prop-strongclassical} hold. $\Box$

\section{Parabolic comparison and qualitative  properties in a dissipative case}\label{s: specificnonlinearity}

In this section we first prove a parabolic comparison principle for mild solutions of
\ref{model_problem}, and then we specialize to dissipative terms of the form $f(u)=-u$ and $f(u)=-u|u|^{q-1}$, $q\in(0,1)\cup(1,+\infty)$. In the linear case $f(u)=-u$ we recover an explicit representation of the unique mild solution in terms of the semigroup generated by $\Delta+\operatorname{id}$ and deduce standard $\ell^p$-decay estimates. In the nonlinear case we  construct suitable one-dimensional barriers. This yields global-in-time upper bounds, positivity of solutions for $q>1$, and finite-time extinction when $q\in(0,1)$.

\subsection{Parabolic comparison}

\begin{theorem}\label{thm:parabolic-comparison-infinite}
Let $G$ be a connected graph satisfying  \textnormal{\ref{IP}}. Fix $1\le p<\infty$, and assume \eqref{eq:Delta(C_c)} if $p>1$. Assume either \textnormal{\ref{F1:monotonicity}}   or \textnormal{\ref{F2:Lipschitz}}. Let $h,g\in L^1_{\mathrm{loc}}([0,T];\ell^p(X,\mu))$ and $u_0,v_0\in\ell^p(X,\mu)$, and let $u,v\in C([0,T];\ell^p(X,\mu))$ be the unique mild solutions of the two problems:
\[
\begin{cases}
    \partial_t u + \mathcal{A} u = h &\quad  \text{ in } (0,T)\times X\\
    u(0)=u_0 &\quad  \text{ in }  X
\end{cases}
\qquad \text{and} \qquad
\begin{cases}
    \partial_t v + \mathcal{A} v = g & \quad  \text{ in } (0,T)\times X\\
    v(0)=v_0 & \quad  \text{ in }  X.
\end{cases}
\]
If $u_0\le v_0$ pointwise on $X$ and $h\leq g$  on $[0,T]\times X$, then
\[
u(t)\leq v(t) \qquad \text{pointwise on $X$ for all $t\in [0, T]$.}
\]

\end{theorem}

\begin{proof}
By Theorem~\ref{th: F1}   the resolvent $J_\lambda:=(\mathrm{id}+\lambda \mathcal{A})^{-1}$  is well-defined on $\Omega_p$, dense in $\mathrm{dom}_p(\mathcal{A})$, and by  Corollary~\ref{cor:min} is order-preserving. Fix $\epsilon>0$. By~\cite[Lemma 4.1]{evans1977nonlinear} there exists an $\epsilon$–discretization $\mathcal{D}_\epsilon$ as in Definition~\ref{def:epsilon-discretization}, with $\lambda_k:=t_k-t_{k-1}\le\epsilon$ and $h_k= h(t_k, x)$. Define the implicit scheme for $\boldsymbol{u}_\epsilon=\{u_k\}_{k=1}^N$ by
\[
(\mathrm{id}+\lambda_k \mathcal{A})u_k = u_{k-1}+ \lambda_k h_k,\qquad u_0=u_0,
\]
and similarly for $\boldsymbol{v}_\epsilon=\{v_k\}_{k=1}^N$ with $g_k$. 

Since $u_{k-1}\le v_{k-1}$ and $h_k\leq g_k$ imply
$u_{k-1}+\lambda_k h_k\leq v_{k-1}+\lambda_k g_k$, then
\[
u_k = J_{\lambda_k}(u_{k-1}+\lambda_k h_k)  \leq J_{\lambda_k}(v_{k-1}+\lambda_k g_k)=v_k,
\]
by induction for all $k$.

Let $u_\epsilon(\cdot)$, $v_\epsilon(\cdot)$ be the associated piecewise-constant interpolants as per equation~\eqref{epsilon_approximation}. By Theorem~\ref{thm:main2}, $u_\epsilon(t)\to u(t)$ and $v_\epsilon(t)\to v(t)$  in $\ell^p(X,\mu)$ for all $t\in [0,T)$ as $\epsilon\to 0$. The pointwise inequality is preserved in the limit, yielding, by continuity, $u(t)\le v(t)$ for all $t\in[0,T]$. 
\end{proof}

\subsection{The linear case $f(u)=-u$}
In this case ${\mathcal A}=\Delta+\operatorname{id}$ is linear and Problem \eqref{model_problem3} reads
\begin{equation}\label{model_problem_LIN}
	\begin{cases}
		\partial_t u + \Delta u + u= h  &  \quad  \text{in } (0,T)\times X,\\
		u(0) = u_0 & \quad \text{in } X\,.
	\end{cases}
\end{equation}
Let $1\leq p <\infty$. Furthermore, if $p>1$ assume that condition~\eqref{eq:Delta(C_c)} holds. By Theorem \ref{th: F1}, the operator $\mathcal{\mathcal A}_{|\Omega_p}$ is $m$-accretive and $\Omega_p$ is dense in $\dom_p(\mathcal{A})$.  

Therefore, all assumptions of \cite[Theorem 5.7]{benilan1988evolution} are satisfied, from which, for all $h \in L^1_{\rm{loc}}([0,T] ; \ell^p\left(X,\mu\right))$ and  $u_0 \in \ell^p\left(X,\mu\right)$, Problem  \eqref{model_problem_LIN} admits a unique mild solution on $[0,T]$ given by 
 $$u(t)=e^{-t(\Delta+\operatorname{id})}u_0+\int_0^t e^{-(t-s)(\Delta+\operatorname{id})} h(s)\,ds\,.$$
Finally, as a consequence of the $\ell^p(X,\mu)$-contractivity of the semigroup associated with $\Delta$ (see \cite[Theorem 2.9]{keller2021graphs}) $u$ satisfies 
 $$
\|u(t)\|_p \leq e^{-t}\|u_0\|_p + \int_0^t e^{-(t-s)}\|h(s)\|_p\, ds.
$$

\subsection{The case $f(u) = -u|u|^{q-1}$ for $q\in(0,1) \cup(1,+\infty)$}\label{ssec:extintion_time}
In this subsection, we discuss the positivity and finite-time extinction  of the mild solution to Problem \eqref{model_problem3} given by Theorem \ref{thm:main2} in the particular case where $f(u) = -u|u|^{q-1}$ with $q\in(0,1) \cup(1,+\infty)$ and $h=0$. Namely, we consider the problem 
\begin{equation}\label{model_problem_q}
\begin{cases}
    \partial_t u + \Delta u + |u|^{q-1}u=0, &  \quad  \text{in } (0,T)\times X,\\
    u(0)=u_0 & \quad \text{in } X.
\end{cases}
\end{equation}

The next theorem  generalizes the  results proved in \cite{chung2011extinction} for finite graphs to infinite graphs and to the more general setting of mild-solutions.
\begin{theorem}\label{thm:extinction-infinite}
Let $G$ be a connected graph. Fix $1\le p<\infty$, and assume \eqref{eq:Delta(C_c)} if $p>1$. For nonnegative $0\leq u_0\in \ell^p(X,\mu)\cap \ell^\infty(X,\mu)$, let $u$ be the unique mild solution to \eqref{model_problem_q} with $q\in(0,1) \cup(1,+\infty)$.

Set $M:=\|u_0\|_{\infty}$. Then $u$ satisfies one of the following
\begin{itemize}
\item[1)] if $q\in(0,1)$, then $0\leq u(t,x) \leq \bigl[M^{1-q}-(1-q)t\bigr]_+^{\frac1{1-q}}$ for all $(t,x)\in (0,+\infty)\times X$, 

\item[2)] if $q\in(1,+\infty)$ and $u_0\not \equiv 0$, then
$ 0<u(t,x) \leq \bigl[\frac{1}{M^{q-1}}+(q-1)t\bigr]^{-\frac1{q-1}}$ for all $(t,x)\in (0,+\infty)\times X$.
\end{itemize}

Therefore, if $q\in (0,1)$
\[
u(t, \cdot)\equiv 0\quad\text{for all }t\ge T_\ast :=\dfrac{M^{1-q}}{1-q}.
\]
\end{theorem}

\begin{remark}\label{rem:construction_approximate_solutions}
    In the proof of Theorem \ref{thm:extinction-infinite} we exploit the fact that the $\epsilon$-approximate solutions $u_\epsilon$ in \eqref{epsilon_approximation} (yielding the mild solution in Theorem \ref{thm:main2}), are obtained as monotone limits of $\epsilon$-approximate solutions $u_\epsilon^{(n)}$ for the corresponding parabolic problem on the sequence of the Dirichlet subgraphs $G_{\operatorname{dir}, n}$ in \eqref{eq:def_Dir_subgraph}, (see the proof of Theorem \ref{th: F1}). In the specific framework of Problem \eqref{model_problem_q} where $h\equiv 0$, an $\epsilon$ discretization $\mathcal{D}_\epsilon$ of $[0,T]$ is a partition of the time interval
	$\left\{ \{t_k\}_{k=0}^N \mid 0= t_0<t_1<\ldots<t_N = T \right\}
	$ with $t_k-t_{k-1} \leq \epsilon$ for every $k=1,\ldots, N$.  Then, the $\epsilon$-approximates subordinate to $\mathcal{D}_\epsilon$ satisfy 
    $$u_\epsilon^{(n)}(0)=u_0=u_\epsilon(0)\,, \quad u_\epsilon^{(n)}(t) = \sum_{k=1}^N  u_k^{(n)}\mathds{1}_{(t_{k-1},t_k]}(t)\leq \sum_{k=1}^N  u_k\mathds{1}_{(t_{k-1},t_k]}(t)= u_\epsilon(t) \quad  \text{for } t\in(0,T]$$
    In particular, since, $u_k^{(n)} \to u_k$ in $\ell^p(X,\mu)$ (as shown in the proof of Theorem \ref{th: F1}),
    $$ \lim_{n \to + \infty} u_\epsilon^{(n)}(t)=u_\epsilon(t) \quad \text{in } \ell^p(X,\mu)\,,\quad  \text{for all } t \in [0,T]\,.$$
   
\end{remark}
In the following subsections we present some preliminary results, respectively for  $q\in(0,1)$  and for  $q\in(1,+\infty)$, which then allow to prove Theorem \ref{thm:extinction-infinite} in Subsection \ref{TH63proof}.

\subsubsection{Upper estimates for the $u_k^{(n)}$'s when $q\in(0,1)$}\label{estimates}
As in the proof of Lemma \ref{lem:bijective-sign-norm} we consider the function $\psi(s)=s-\lambda f(s)$, which here reads 
\begin{equation}\label{eq:vartheta}
    \psi_{q,\lambda} (s) = s + \lambda s|s|^{q-1}, \qquad  q\in (0,1), \; \lambda >0.
\end{equation}

Since $\psi_{q,\lambda}: \R \to\R$ is strictly increasing, its inverse $\psi_{q,\lambda}^{-1}: \R \to\R$ is well-defined and strictly increasing as well.

Fix  $0 \leq u_0 \in \ell^p(X,\mu) \cap \ell^\infty(X,\mu)$ and let $M = \|u_0\|_\infty$.  
\\

\begin{lemma}\label{lem:supersolution}
    Let $q\in(0,1)$ and let $F: \R \to\R$ be defined  by $F(s)= s |s|^{q-1}$.
Given  any partition $0=t_0<t_1<\ldots<t_N= T$ with $\lambda_k:=t_k-t_{k-1}$, define recursively $\theta_k$ by 
\begin{equation}\label{eq:lem:supersolution}
\begin{cases}
\theta_k+\lambda_k\,F(\theta_k)=\theta_{k-1} &\mbox{for } k=1,\ldots, N,\\
 \theta_0=M\geq 0.
\end{cases}
\end{equation}
Let  also $\theta_\epsilon:[0,T]\to\R$ be the piecewise-constant interpolant  $\theta_\epsilon(t)=\theta_k$ for $t\in(t_{k-1},t_k]$, as per equation~\eqref{epsilon_approximation}, where $\epsilon:=\max_k\lambda_k$, and let
\[
\theta(t)=\bigl[M^{1-q}-(1-q)t\bigr]_+^{\frac1{1-q}}.
\]
Then 
$$
\lim_{\epsilon \to 0}|\theta_\epsilon(t) - \theta (t)|=0 \qquad \mbox{for every } t\in [0, T].
$$
\end{lemma}
\begin{proof}

Define $T_\ast:=\frac{M^{1-q}}{1-q}$.
For $t<T_\ast$ we have $\theta(t)>0$ and
\[
\dot \theta(t)= -\bigl[M^{1-q}-(1-q)t\bigr]^{\frac{q}{1-q}}
=-F(\theta(t)).
\]
For $t\ge T_\ast$, it holds $\theta(t)\equiv 0$ and it also solves $\dot \theta+ F(\theta)=0$. Hence $\theta$ is a global strong solution of
\begin{equation}
\label{eq:ode}
\begin{cases}
\dot \theta + F(\theta)=0, \\
\theta(0) = M.
\end{cases}
\end{equation}
The sequence $\{\theta_k\}_{k=1}^N$ defined in \eqref{eq:lem:supersolution} is well-defined and satisfies the relation $\theta_k = \psi_{q,\lambda_k}^{-1}(\theta_{k-1})$. In particular, the implicit Euler system~\eqref{eq:lem:supersolution} admits a solution and the corresponding interpolant $\theta_\epsilon$ is an $\epsilon$-approximation to problem~\eqref{eq:ode}. 
Using the fact that $\operatorname{sgn}(F(\theta) - F(\eta)) =  \operatorname{sgn}(\theta - \eta)$ by monotonicity, 
it is not difficult to check that $F$ is an $m$-accretive operator on $\mathbb{R}$. Thus, by \cite[Theorem 4.1]{barbu2010nonlinear}, $\theta_\epsilon$ converges to the unique mild solution to problem~\eqref{eq:ode}.   Thus,  uniqueness of mild solutions for $m$-accretive $F$, yields $| \theta_\epsilon (t) - \theta(t)| \to 0$ as $\epsilon \to 0$, for every fixed $t$.
\end{proof}

For a finite connected subset $X_n \subseteq X$, consider the Dirichlet subgraph $G_{\textnormal{dir},n}$ as defined in~\eqref{eq:def_Dir_subgraph}. Fix  any partition $0=t_0<t_1<\ldots<t_N= T$ with $\lambda_k:=t_k-t_{k-1}$, and define the finite sequence $\{u_k^{(n)}\}_{k=1}^N$ recursively by
\begin{equation}\label{eq:Euler_dirn}
\begin{cases}
   (\operatorname{id} + \lambda_k (F + \Delta_{\textnormal{dir}, n}))u_k^{(n)}  = u_{k-1}^{(n)} \quad \text{in }X_n & \text{for } k = 1,\dots,N,\\[6pt]
   u^{(n)}_0 = \boldsymbol{\pi}_n u_0\,.
\end{cases}
\end{equation}
This sequence is well-defined for every $k$, as shown by Lemma~\ref{lem:bijective-sign-norm} (recall that $F$ satisfies condition \textnormal{\ref{F1:monotonicity}} ). Furthermore, the $\ell^{\infty}$ norm of the functions $u_k^{(n)}$ satisfies the recurrence estimate provided below. 

\begin{lemma}\label{lem:upper_bound_k}
Consider the inverse $\psi_{q,\lambda}^{-1}: \R \to\R$ of the map introduced in \eqref{eq:vartheta} and the sequence $\{u_k^{(n)}\}_{k=1}^N$ given in \eqref{eq:Euler_dirn}. For each $k = 1,\dots,N$, define
\begin{equation}
\beta_k^{(n)} = \psi_{q,\lambda_k}^{-1}(M_{k-1}^{(n)}), \quad \text{where } M_{k-1}^{(n)} = \|u_{k-1}^{(n)}\|_\infty.
\end{equation}
Then, it holds that
\[
\|u_{k}^{(n)}\|_\infty \leq \beta_k^{(n)}.
\]
\end{lemma}
\begin{proof}
Since $u_0 \geq 0$, by a recursive argument and Theorem~\ref{thm:min_principle}-$(a)$, which obviously holds for finite graphs, $u_k^{(n)}\geq 0$. Let $x_k\in X_n$ be such that $u_k^{(n)}(x_k)=\|u_k^{(n)}\|_\infty$.  
Since $G_{\mathrm{dir},n}$ is finite, such a node exists and, since $u_k^{(n)}$ attains its maximum at $x_k$ 
$\Delta_{\mathrm{dir},n}u_k^{(n)}(x_k)\ge0$. 

Evaluating \eqref{eq:Euler_dirn} at $x_k$ gives
$$
u_k^{(n)}(x_k)+\lambda_k
\big(u_k^{(n)}(x_k)^{q}+\Delta_{\mathrm{dir},n}u_k^{(n)}(x_k)\big)
=u_{k-1}^{(n)}(x_k)\leq M_{k-1}^{(n)},
$$
which, using the notation in~\eqref{eq:vartheta}, can be rewritten as 
$$
\psi_{q,\lambda_k}\big(u_k^{(n)}(x_k)\big) + \lambda_k\Delta_{\mathrm{dir},n}u_k^{(n)}(x_k)\leq M_{k-1}^{(n)}.
$$
Discarding the nonnegative graph Laplacian term yields
\[
\psi_{q,\lambda_k}\big(u_k^{(n)}(x_k)\big)\le M_{k-1}^{(n)},
\]
which implies,
$$
u_k^{(n)}(x_k)\le \psi_{q,\lambda_k}^{-1}(M_{k-1}^{(n)}) = \beta_k^{(n)},
$$
thus concluding the proof.
\end{proof}

Finally we prove upper estimates of the $u_k^{(n)}$ in terms of the $\theta_k$ given in \eqref{eq:lem:supersolution}.

\begin{lemma}\label{lem:upper_barrier_bound}
Let $\{u_k^{(n)}\}_{k=1}^N$ be the sequence given in \eqref{eq:Euler_dirn}. For all $k=0,\ldots, N$ and $n\in \N$, the following inequality holds 
$$\|u_k^{(n)}\|_\infty \leq \theta_k,$$ where the sequence $\{\theta_k\}_{k=1}^N$ is defined as in \eqref{eq:lem:supersolution}.
\end{lemma}

\begin{proof}
The scalar sequence $\{\theta_k\}_{k=1}^N$ from Lemma~\ref{lem:supersolution} satisfies the recursive relation
\[
\theta_k + \lambda_k\theta_k^{q} = \theta_{k-1} \mbox{ for } k=1,\ldots, N, \quad \text{with} \quad \theta_0 = \|u_0\|_\infty.
\]
Equivalently, $\theta_k = \psi_{q,\lambda_k}^{-1}(\theta_{k-1})$ for $k=1,\ldots, N$.

We proceed by induction on $k$. For the base case $k=0$, clearly,
\[
\|u_0^{(n)}\|_\infty \leq \|u_0\|_\infty = \theta_0.
\]

Assume now by induction that $M_{k}^{(n)} = \|u_{k}^{(n)}\|_\infty \leq \theta_k$ for every $n$. From Lemma~\ref{lem:upper_bound_k}, we already have that
\[
M_{k+1}^{(n)} = \|u_{k+1}^{(n)}\|_\infty \leq \beta_{k+1}^{(n)}.
\]
By applying the induction hypothesis and using the monotonicity of $\psi_{q,\lambda_k}^{-1}$, we obtain
\[
\|u_{k+1}^{(n)}\|_\infty \leq \beta_{k+1}^{(n)} = \psi_{q,\lambda_{k+1}}^{-1}(M_{k}^{(n)}) \leq \psi_{q,\lambda_{k+1}}^{-1}(\theta_k) = \theta_{k+1}.
\]

Thus, by induction, for all $k = 0, \ldots, N$,
\[
\|u_k^{(n)}\|_\infty \leq \theta_k,
\]
independently of $n$.
\end{proof}

\subsubsection{Positivity of solutions in the case $q\in (1,+\infty)$}\label{positivity}

Let $u_0 \geq 0$ with $u_0\not \equiv 0$. For $\epsilon > 0$, let $u_\epsilon$ be the $\epsilon$-approximate solutions in \eqref{epsilon_approximation} defined on the entire graph $G$.   As explained in Remark \ref{rem:construction_approximate_solutions}, by construction, we have
\begin{equation}\label{lim_n_eps}
u_\epsilon^{(n)}(t,x) \leq u_\epsilon (t,x) \quad \text{and} \quad u_\epsilon(t,x) = \lim_{n \to \infty} u_\epsilon^{(n)}(t,x), \quad \text{for every } (t,x)\in [0,T]\times X\,.
\end{equation}

For every fixed $n$, consider on the finite graph $G_{\operatorname{dir},n}$ the ``mild'' solution $u^{(n)}$ to which the sequence $\{u_\epsilon^{(n)}\}_{\epsilon}$ converges in $\ell^p(X,\mu)$ for all $t\in [0,T]$. For finite graphs, mild, strong and classic solutions are equivalent. Since $q\in (1,+\infty)$, from~\cite[Theorems 4.7-4.8]{chung2011extinction}, $u^{(n)} >0$ in $X_n$ and therefore there exists $\epsilon = \epsilon(n)$ such that for every $0<\epsilon \leq \epsilon(n)$ it holds
$$
u^{(n)}_\epsilon (t,x) \geq\frac{1}{2}  u^{(n)} (t,x) > 0 \quad \text{for all } (t,x)\in [0,T]\times  X_n. 
$$

 Fix $x\in X$ and $t\in[0,T]$. Let $X_{n_0}$ be the first element of the exhaustion $\{X_n\}_n$ such that $x\in X_{n_0}$. By the inequality above, and since the convergence of the sequence $\{u_\epsilon\}_{\epsilon}$ to the mild solution $u$ in $\ell^p(X,\mu)$ implies the pointwise convergence for every fixed $t\in [0,T]$ and $x\in X$ , we get the following pointwise estimate:
$$
u(t,x) = \lim_{\epsilon\rightarrow 0} u_{\epsilon} (t,x)\geq \lim_{\epsilon\rightarrow 0} u^{(n_0)}_{\epsilon} (t,x) \geq \frac{1}{2} u^{(n_0)}(t,x) >0.
$$
This shows the positivity of the mild solution $u$ for all $(t,x)\in [0,T]\times X$.

\subsubsection{Proof of Theorem \ref{thm:extinction-infinite}}\label{TH63proof}

We first prove the statement $1)$. Fix $\epsilon > 0$. As in Remark \ref{rem:construction_approximate_solutions} and in Section \ref{positivity}, we denote by $\{u_\epsilon^{(n)}\}_n$ the sequence of $\epsilon$-approximate solutions constructed on each finite Dirichlet subgraph $G_{\mathrm{dir},n}$, and by $u_\epsilon$ the $\epsilon$-approximate solution defined on the entire graph $G$.

Let $\theta_\epsilon$ be the piecewise-constant interpolant defined in Lemma \ref{lem:supersolution}. Lemma~\ref{lem:upper_barrier_bound} ensures that $0 \leq u_\epsilon^{(n)}(t,x) \leq \theta_\epsilon(t)$ for all $n$ and all $(t,x)\in [0,T]\times  X_n$, which by \eqref{lim_n_eps} implies
$$
0\leq u_\epsilon(t,x) \leq \theta_\epsilon(t)$$
for all $t\in [0,T]$, locally uniformly in $x\in X$. For fixed $t\in [0,T]$, given that the mild solution $u$ satisfies
$$
\lim_{\epsilon \to 0} \| u(t) - u_\epsilon(t) \|_p = 0 \,,
$$
it follows that
$$
0\leq u(t,x) = \lim_{\epsilon \to 0} u_\epsilon(t,x) \leq \lim_{\epsilon \to 0} \theta_\epsilon(t) = \theta(t) \quad \text{for all } x \in X,
$$
where the last equality comes from Lemma~\ref{lem:supersolution}.

Next we turn to the statement $2)$. In this case $q\in(1,+\infty)$ and we need to replace $\theta$ with 
$$
\theta(t) = \bigl[\frac{1}{M^{q-1}}+(q-1)t\bigr]^{-\frac1{q-1}}, \qquad M = \|u_0\|_\infty.
$$
Then, all estimates in the proof of $1)$ follow in the same way with minor changes. The positivity has already been proved in Section \ref{positivity}. $\Box$

\begin{remark}\label{rem:caso_con_segno}
We observe that, under the assumption of Theorem \ref{thm:main2}, the proof of extinction in finite time when $q\in(0,1)$ can be adapted to the case of a sign changing initial datum $u_0 \in \ell^p(X,\mu)\cap \ell^\infty(X,\mu)$  in the following way. Consider the sequences $\{u_k^{(n),+}\}_k$ and $\{v_k^{(n),-}\}_k$ defined recursively by \eqref{eq:Euler_dirn} with starting value, respectively $u_0^+\geq 0$ and $-u_0^- \geq 0$ (see \eqref{pm}). Namely,
\begin{equation}\label{eq:Euler_dirn+}
\begin{cases}
   (\operatorname{id} + \lambda_k (F + \Delta_{\textnormal{dir}, n}))u_k^{(n),+}  = u_{k-1}^{(n),+} & \text{for } k = 1,\dots,N,\\[6pt]
   u^{(n),+}_0 = \boldsymbol{\pi}_n u_0^+ 
\end{cases}
\end{equation}
and
\begin{equation}\label{eq:Euler_dirn-}
\begin{cases}
   (\operatorname{id} + \lambda_k (F + \Delta_{\textnormal{dir}, n}))v_k^{(n),-}  = v_{k-1}^{(n),-} & \text{for } k = 1,\dots,N,\\[6pt]
   v^{(n),-}_0 =- \boldsymbol{\pi}_n u_0^-.
\end{cases}
\end{equation} 
Applying Lemma \ref{lem:upper_barrier_bound}, we infer that $u_{k}^{(n),+}\leq \theta_k$ and $v_{k}^{(n),-}\leq \theta_k$. Furthermore, since $\boldsymbol{\pi}_n u_0^{+} \geq \boldsymbol{\pi}_n u_0$ and $-\boldsymbol{\pi}_n u_0^{-} \geq -\boldsymbol{\pi}_n u_0$, from Corollary \ref{cor:min} applied recursively, we also get that $u_{k}^{(n)} \leq u_{k}^{(n),+}$ and $u_{k}^{(n)} \geq - v_{k}^{(n),-}$. Finally, we set $u_{k}^{(n),-}=-v_{k}^{(n),-}$ and we conclude that $u_{k}^{(n),-} \leq u_{k}^{(n)} \leq u_{k}^{(n),+}$ for $k=0, \ldots, N$, and, for the corresponding piecewise-constant interpolants,

$$-\theta_\epsilon(t) \leq u_\epsilon^{(n),-} (t,x) \leq   u_\epsilon^{(n)} (t,x) \leq u_\epsilon^{(n),+} (t,x)  \leq \theta_\epsilon(t).$$ 
Hence the conclusion follows as in the proof of Theorem~\ref{thm:extinction-infinite}-(1). 

\end{remark}

\textbf{Acknowledgments.} We thank Rados{\l}aw K. Wojciechowski   for valuable comments and helpful suggestions. E. Berchio, M. Vallarino and A. G. Setti are members of the Gruppo Nazionale per l'Analisi Matematica, la Probabilit\`a e le loro Applicazioni (GNAMPA, Italy) of the Istituto Nazionale di Alta Matematica (INdAM, Italy). The research of E. Berchio was carried out within the PRIN 2022 project 2022SLTHCE - Geometric-Analytic Methods for PDEs and Applications GAMPA, funded by European Union - Next Generation EU within the PRIN 2022 program (D.D. 104 - 02/02/2022 Ministero dell'Universit\`a e della Ricerca). D. Bianchi is supported by the Startup Fund of Sun~Yat-sen~University.  This manuscript reflects only the authors' views and opinions and the Italian Ministry cannot be considered responsible for them.

\printbibliography
\end{document}